\documentclass[10pt, twoside, reqno]{article}
\usepackage{amssymb, amsmath,amsthm}
\usepackage{graphicx,color}

\newtheorem{thm}{Theorem}[section]
\newtheorem{theorem}[thm]{Theorem}
\newtheorem{definition}[thm]{Definition}
\newtheorem{lemma}[thm]{Lemma}
\newtheorem{corollary}[thm]{Corollary}
\newtheorem{conjecture}[thm]{Conjecture}
\newtheorem{proposition}[thm]{Proposition}

\newcommand{\Q}{{\mathcal{Q}}}
\newcommand{\F}{{\mathcal{F}}}

\newcommand{\deletedred}{{\color{red}{\tiny{$\odot$ \kern-1em $\star$}}}}
\newcommand{\bridgered}{{\color{red}{\footnotesize{$\circ$ \kern-.85em $\star$}}}}
\newcommand{\bridgegreen}{{\color{green}{\footnotesize{$\circ$}}}}
\newcommand{\deletedgreen}{{\color{green}{\tiny{$\ominus$}}}}
\newcommand{\firsttermgreen}{{\color{green}{\tiny{$\triangle$}}}}
\newcommand{\firsttermred}{{\color{red}{\tiny{$\triangle$ \kern-1em $\star$}}}}

\newcommand{\secondtermgreen}{{\color{green}{\tiny{$\square$}}}}
\newcommand{\secondtermred}{{\color{red}{\tiny{$\square$ \kern-.9em \raise.3ex\hbox{$\star$}}}}}

\newcommand{\twoplates}{{
\put(30,120){\line(1,0){50}} \put(0,70){\line(1,0){50}} 
\put(0,70){\line(3,5){30}} \put(50,70){\line(3,5){30}} 
\put(30,50){\line(1,0){50}} \put(0,0){\line(1,0){50}} 
\put(0,0){\line(3,5){30}} \put(50,0){\line(3,5){30}}}}

\newcommand{\twoplatescl}{{
\put(30,90){\line(1,0){50}} \put(0,50){\line(1,0){50}} 
\put(0,50){\line(3,4){30}} \put(50,50){\line(3,4){30}} 
\put(30,40){\line(1,0){50}} \put(0,0){\line(1,0){50}} 
\put(0,0){\line(3,4){30}} \put(50,0){\line(3,4){30}}}}

\begin{document}

\title{Path coverings with prescribed ends in faulty hypercubes}

\author{Nelson Casta\~neda\\
Department of Mathematical Sciences,\\
Central Connecticut State University,\\
1615 Stanley Street, New Britain, CT 06050, USA\\
E-mail: castanedan@ccsu.edu
\and
Ivan S. Gotchev\\
Department of Mathematical Sciences,\\
Central Connecticut State University,\\
1615 Stanley Street, New Britain, CT 06050, USA\\
E-mail: gotchevi@ccsu.edu}

\date{}

\maketitle


\renewcommand{\thefootnote}{}

\footnote{2010 \emph{Mathematics Subject Classification}: Primary 05C38; 
05C45; Secondary 68R10; 68M10; 68M15.}

\footnote{\emph{Key words and phrases}: Hypercube, path covering, prescribed ends of a path, Hamiltonian path, Hamiltonian cycle.}

\renewcommand{\thefootnote}{\arabic{footnote}}
\setcounter{footnote}{0}

\begin{abstract}
We discuss the existence of vertex disjoint path coverings with
prescribed ends for the $n$-dimensional hypercube with or without
deleted vertices. Depending on the type of the set of deleted
vertices and desired properties of the path coverings we establish
the minimal integer $m$ such that for every $n \ge m$ such path
coverings exist. Using some of these results, for $k \le 4$, we
prove Locke's conjecture that a hypercube with $k$ deleted
vertices of each parity is Hamiltonian if $n \ge k +2.$ Some of
our lemmas substantially generalize known results of I. Havel and
T. Dvo\v{r}\'{a}k. At the end of the paper we formulate some
conjectures supported by our results.
\end{abstract}

\section{Introduction}

The $n-$dimensional binary hypercube $\Q_n$ is the graph whose
vertex set $\mathcal{V}(\Q_n)$ consists of  all binary sequences
of length $n$ and whose edge set $\mathcal{E}(\Q_n)$ consists of
all pairs of binary sequences that differ in exactly one position.
In recent years some attention has been given to the problem of
finding Hamiltonian cycles or maximal cycles in the
$n-$dimensional binary hypercube $\Q_n$ with faulty vertices or
with faulty edges.

In \cite{parkhomenko} Parkhomenko illustrates some techniques of
constructing cycles without faulty edges or vertices in
low dimensional hypercubes. His methods rely on a classification of
Hamiltonian cycles for hypercubes of dimension $4$ or less.

Caha and Koubek \cite{cahakoubek} and Dvo\v{r}\'{a}k \cite{dovarak}
have addressed the problem of prescribing a set of
edges $ \mathcal{P} $ through which a Hamiltonian cycle in $\Q_n$ must pass.
The best theorem in this direction known to us is the following:

\begin{theorem}[Dvo\v{r}\'{a}k \cite{dovarak}] Let $ {\mathcal P}$ be a
set of edges in $\Q_n$ such that each connected component of the subgraph
generated by ${\mathcal P}$ is a simple path. If the cardinality of
$ {\mathcal P}$ is less than or equal to $2n-3,$ then there exists a
Hamiltonian cycle in $\Q_n$ that passes through each edge in ${\mathcal P}.$
\end{theorem}

Dvo\v{r}\'{a}k's proof uses two lemmas about covering the vertices
of $\Q_n$ by vertex disjoint paths with prescribed ends. The first
one, called Havel's lemma, states that given any two vertices of
opposite parity in $\Q_n ,$ with $n \ge 1,$ there exists a
Hamiltonian path with these two vertices as endpoints \cite[Proposition 2.3]{havel}.
Dvo\v{r}\'{a}k generalizes this lemma as follows:

\begin{lemma}[Dvo\v{r}\'{a}k \cite{dovarak}] Let $n \ge 2,$ $a_1,a_2$
be two distinct vertices of the same parity, and $b_1,b_2$ be two
distinct vertices of the opposite parity in the hypercube $\Q_n.$
Then there exist two vertex-disjoint paths, one joining $a_1$ to
$b_1$ and the other joining $a_2$ to $b_2,$ such that each vertex
of $\Q_n$ is contained in one of these paths.
\end{lemma}

One of the main ingredients in the proof of Dvo\v{r}\'{a}k's
theorem is the existence of a covering of the vertices of $\Q_n$
by vertex disjoint  paths with prescribed end vertices. In this
article we address the existence of such path coverings with
prescribed end vertices for the hypercube with or without
deleted vertices. More specifically, we investigate what is
the minimal dimension $m$ of the hypercube $\Q_m$ such that for
every $n \geq m$ and every set $\mathcal{F}$ of $M\geq 0$ deleted
vertices from $\Q_n$ such that the absolute value of the
difference of the numbers of the deleted vertices of the two
parities is $C$, there exists a path covering of $\Q_n -
\mathcal{F}$ with $N$ paths whose end vertices are with different
parity and $O$ paths whose end vertices are of the same parity,
where all of the end vertices of these paths belong to an
arbitrary set of non-deleted vertices. The exact meaning of these
words can be found in Section 2 where more precise definitions are
given including the definition of the symbol $[M,C,N,O]$  that
represents the number $m$ mentioned above.

The main results of this paper are contained in the last 4
sections. Section 3 deals with special cases where the numbers
$M$, $C$, $N$, and $O$ are small and in many of those cases we use
pictorial proofs. In Section 4 we use words to represent paths in
the proofs and we study cases of larger numbers of $M$, $C$, $N$,
or $O$. In particular, in that section, we generalize
Dvo\v{r}\'{a}k's lemma (see Lemma \ref{threecoveringlemma}).
Section 5 contains general results that allow us to establish
connections between different values of $[M,C,N,O]$. These three
sections also contain, for $k \le 4$, a proof of Locke's
conjecture that a hypercube with $k$ deleted vertices of each
parity is Hamiltonian if $n \ge k +2.$ In Section 6 we state some
conjectures supported by our results and we give some concluding
remarks. Appendix \ref{2pathcoveringsofQ4} contains a proof of a
claim for $n=4$ that we found difficult to verify by inspection.
In a table in Appendix \ref{B} we summarize many of the results
contained in this paper.

\section{Some definitions}

To simplify the explanations that follow we introduce the following
terminology and conventions.
\emph{A path covering of a graph} is a set of vertex disjoint
paths that cover all the vertices of a given graph. \emph{$k-$path
covering} is a path covering by exactly $k$ paths. Sometimes we
call the end vertices of a path \emph{ends} or \emph{terminals.}
A vertex of $\Q_n$ is called \emph{even} (\emph{odd}) if it has an even
(odd) number of $1's.$ A transformation that changes the values of
a fixed  entry for all the vertices of $ \Q_n$ induces an
automorphism of the hypercube that sends even vertices to odd
vertices and vice versa. Therefore, any  statement about $\Q_n$ in
terms of even and odd vertices has an equivalent dual statement
obtained when the references to even and odd vertices are
interchanged.
For convenience, we call the vertices of one parity \emph{red} and
the vertices of the opposite parity \emph{green} without
specifying which are even and which are odd.

A \emph{fault} $\mathcal{F}$ in $\Q_n$ is a set of deleted vertices.
The \emph{mass} $M$ of a fault $\F$ is the total
number of vertices in the fault. The \emph{charge} $C$ of a fault is
the absolute value of the difference between the
number of red vertices and the number of green vertices. We say that a
fault is \emph{neutral} if its charge is zero.
When the endpoints of a path are of the same parity we say that the path
is \emph{charged}; otherwise the path is \emph{neutral.}
Regarding a pair of vertices we say that the \emph{pair is charged} if
the two elements in the pair are of the same parity and that the
\emph{pair is neutral} if the two elements are of opposite parity. If the two
elements of a charged pair of vertices are red (green) we say that
the \emph{pair is red (green)}.

Let $M$ be any nonnegative even number and let $\mathcal{A}_M$ be the set of
positive integers $m$ with the property that if $n \ge m$ then $\Q_n - \F$
is Hamiltonian for every neutral fault $\F$ of mass $M$ in $\Q_n.$ The set
$\mathcal{A}_M $ is nonempty (see \cite{lockestong}). We denote by $[M]$ the
smallest integer in this set. It is clear that $[0]=2$ since $\Q_n$ is
Hamiltonian if $n \ge 2$, and $[2k]\ge k+2$ since if $k$ vertices
adjacent to a given vertex are removed from $\Q_{k+1}$ then the resulting
graph is not Hamiltonian. In Problem 10892 of {\it The American Mathematical
Monthly} \cite{locke} S. Locke conjectures that $[2k]=k +2 $ for every
nonnegative integer $k.$  A proof of $[2]=3$ is contained in
\cite{lockestong} and a proof of $[4]=4$ was known to S. Locke
(personal communication). To the best of our knowledge Locke's conjecture in
its full generality remains unsolved. In Lemmas \ref{locke2}, \ref{locke3},
and \ref{locke4}, we prove that $[2k] = k +2 $ for $k=2,3,4.$

Let $r(\F)$ be the number of red vertices and  $g(\F)$ be the
number of green vertices in a fault $\F$  of $\Q_n.$ Let also
$\mathcal{E}$ be a set of disjoint pairs of vertices of $\Q_n$,
$r(\mathcal E)$ be the number of red pairs in $\mathcal{E}$, and
$g(\mathcal E)$ be the number of green pairs in $\mathcal{E}.$ We
say that \emph{the set of pairs $\mathcal{E}$ is in balance with
the fault $\F$} if all the vertices in the elements of
$\mathcal{E}$ are from $\Q_n - \F$ and $r(\F) - g(\F) =
g(\mathcal{E}) - r(\mathcal{E}).$ Since $\Q_n$ is a bipartite
graph with the set of even vertices  and the set of odd vertices
as partite sets, a necessary condition for a set $\mathcal{E}$ of
pairs of vertices to be the set of endpoints of a path covering of
$\Q_n - \F$ is that $\mathcal{E}$ to be in balance with $\F.$

\begin{definition} Let $M, C, N, O $ be nonnegative integers and $\F$
be a fault of mass $M$ and charge $C$ in $\Q_n.$  We say that \emph{one can freely prescribe ends for a path covering of
$\Q_n- \F$ with $N$ neutral paths and $O$ charged paths} if
\begin{itemize}
\item[(i)] there exists at least one set $\mathcal{E}$ of disjoint pairs of vertices that is in balance with $\F$
    and  contains exactly $N$ neutral pairs and $O$ charged pairs; and \item[(ii)] for every set $\mathcal{E}$ of
    disjoint pairs of vertices that is in balance with $\F$ and contains exactly $N$ neutral pairs and $O$ charged
    pairs there exists a path covering of $\Q_n- \F$ such that the set of pairs of end vertices of the paths in the
    covering coincides with $\mathcal{E}.$
\end{itemize}
\end{definition}

It is easy to see that if in $\Q_n$ there exists a fault $ \F$  of mass $M$ and charge $C,$  and  a set of pairs of
vertices $ \mathcal{E}$ that is  in balance with $\F$ and contains exactly $N $ neutral pairs and $ O$ charged pairs,
then $ 2^n \ge M + C + 2N + 2 O. $

\begin{definition}
Let $\mathcal{A}_{M,C,N,O}$ be the set of nonnegative integers $m$ such that
\begin{itemize}
\item[(i)] $m \ge \log_2{[M + C + 2N + 2O]};$ and \item[(ii)] for every $n \ge m$ and  for every fault $\F$ of mass
    $M$ and charge $C$ in $ \Q_n$ one can freely prescribe ends for a path covering of $\Q_n - \F$ with $N$ neutral
    paths and $O$ charged paths.
\end{itemize}

We let $[M,C,N,O] $ denote the smallest element in $\mathcal{A}_{M,C,N,O}$ if
this set is nonempty.
\end{definition}

For example, Havel's lemma quoted above is the statement
$[0,0,1,0]=1$ and Dvo\v{r}\'{a}k's lemma is the statement $[0,0,2,0]=2.$

\section{Some cases of small faults or small sets of prescribed end vertices}

In the statements below, since only a few vertices are deleted from
$\Q_{n+1}$ and we are looking for path coverings
with just a few paths, it is convenient to illustrate the proofs by using
diagrams. In these diagrams the hypercube
$\Q_{n+1}$ is viewed as two copies of the $n-$dimensional hypercube which we
call \emph{top plate} and \emph{bottom plate}
and we denote by $\Q_{n+1}^{top}$ and $\Q_{n+1}^{bot}$, respectively. The
edges connecting the two plates are called
\emph{bridges}. We mark on the diagrams only the vertices that are relevant
for the proof. To distinguish their colors
(parity) we mark the red vertices with stars and leave the green ones
unmarked. The prescribed ends of each path are
represented by the same geometric figure (triangle, square, etc.) and for
different paths we use different figures. The
deleted vertices are represented by big circles with a star inside if they
are red or a minus inside if they are green.
For the proof of a given lemma we usually produce connections on the plates
that are guaranteed by previous lemmas or by
an induction hypothesis and then we use bridges to connect paths from the top
plate to paths from the bottom plate.
Sometimes the paths from a plate are cut at certain places and the cut points
are connected to the other plate by
\emph{bridges}. In such cases we say that \emph{we perform surgery}. The
vertices at which we do cuts are represented by
tiny circles. The variables $r,r_1,r_2,\dots$ are reserved to represent red
vertices and the variables $g,g_1,g_2,\dots
$ are reserved to represent green vertices.

The following lemma that qualifies $ \Q_n$ as a hyper-Hamilton
laceable graph was proved by Lewinter and Widulski \cite[Corollary
4]{lewinter}.

\begin{lemma} $([1,1,0,1]=2)$ Let $n\ge 2$ and $d$ be any vertex in $\Q_n.$
Then one can freely prescribe ends for a charged Hamiltonian path
of $\Q_n- \{d\}.$
\end{lemma}

Corollary \ref{refinedhavel} below is a refinement of Havel's lemma and
follows directly from $[0,0,2,0]=2$ and $[1,1,0,1]=2$. It also appears as
Corollary 3.4 in \cite{dovarak} and therefore is given here without proof.

\begin{corollary} \label{refinedhavel} Let $n \ge 2$, $r$ and $g$ be a
red and a green vertex in $\Q_n,$ and $e$ be an edge different from
$\{r,g\}.$ Then there exists a Hamiltonian path of $\Q_n$
that connects $r$ to $g$ and passes through $e.$
\end{corollary}

The following lemma  is a solution to the first part of Problem
10892 proposed by S. Locke in {\it The American Mathematical
Monthly} \cite{locke}. For the solution published in {\it The
Monthly} see \cite{lockestong}. We present a different proof.

\begin{lemma}\label{locke1} $([2]=3)$ If $n \ge 3$ then $\Q_n - \F$ is
Hamiltonian for any neutral fault $\F$ of mass $2.$
\end{lemma}

\begin{proof} Produce two plates that separate the deleted vertices $r$ and $g$
and assume that the deleted red vertex $r$ is on the top plate.
Find two bridges with green vertices on the top plate that do not
contain the deleted vertices. Use $[1,1,0,1]=2$ to produce a
Hamiltonian path of $\Q_n^{top} - \{r\}$ that connects the top
vertices of the bridges. Use $[1,1,0,1]=2$ to produce a
Hamiltonian path of $\Q_n^{bot} - \{g\}$ that connects the lower
vertices of the bridges. The paths produced on the plates
connected by the bridges form the desired Hamiltonian cycle in
$\Q_n - \F.$
\end{proof}

\begin{lemma}\label{twodifferentedges}
Let $n\ge 2,$ $r$ be a red vertex and $g_1, g_2$ be two green vertices in
$\Q_n.$ Then there are at least $n-1$ Hamiltonian paths of $\Q_n - \{r\}$ that
connect $g_1$ to $g_2$, all starting with different edges.
\end{lemma}

\begin{proof} The proof is by induction. The statement is obvious for $n=2$.
When $n=3$ there are only two cases to consider: $r$ belongs to the same two
dimensional subcube that contains $g_1$ and $g_2$ and $r$ does not belong to
it. In each one of these cases it is routine to construct the required two
paths.

Now let $n \ge 4.$ Produce two plates to separate the two green vertices.

\emph{Case 1.} $r$ and $g_1$ are on the top plate and $g_2$ is on the bottom
plate.

Let $g$ be any green vertex on the top plate different from $g_1$ and $r_1$ be
the vertex of $Q_n^{bot}$ that is adjacent to $g.$ By the induction hypothesis
there are at least $n-2$ Hamiltonian paths of $\Q_n^{top}-\{r\}$ that
connect $g_1$ to $g$ all starting with different edges from $g_1.$  Extend
each of these paths to produce a Hamiltonian path of $\Q_n- \{r\}$ that
connects $g_1$ to $g_2$ by adding the bridge $\{g,r_1\}$ and then a
Hamiltonian path of $\Q_n^{bot} $ that connects $r_1$ to $g_2.$ The latter
path exists since $[0,0,1,0]=2.$ Finally, let $r_2$ be the vertex
of $\Q_n^{bot}$ that is adjacent to $ g_1.$ We produce a Hamiltonian path of
$\Q_n - \{r\}$ that connects $ g_1$ to $g_2$ and starts with the bridge
$\{g_1,r_2\}$ as follows. Produce a Hamiltonian cycle of
$\Q_n^{top} -\{g_1,r\}.$
Such cycle exists since $[2]=3.$ Cut this Hamiltonian cycle at two consecutive
vertices whose adjacent vertices on $\Q_n^{bot}$ are a green vertex
$g_3\ne g_2$ and a red vertex $r_3 \neq r_2.$ Such consecutive vertices exist
since the length of the cycle is at least six. Produce a $2-$path covering of
$\Q_n^{bot} $ with one path connecting $r_2$ to $g_3$ and the other connecting
$r_3$ to $ g_2.$ Such path covering exists because $[0,0,2,0]=2.$ We obtain the
desired Hamiltonian path of $\Q_n -\{r\}$ by adding to the pieces so far
produced the bridge $ \{g_1, r_2\}.$

\emph{Case 2.} $r$ and $g_2$ are on the top plate and $g_1$ is on the bottom
plate.

We can assume that $r$ and $g_1$ are not adjacent; otherwise, we could
separate $r$, $g_1$, and $g_2$ as in Case 1. Let $r_1$ be the neighbor of
$g_1$ on the top plate, $g_3\neq g_2$ be any green vertex on the top plate,
$r_2$ be the neighbor of $g_3$ on the bottom plate, and $g_4\neq g_1$ be
adjacent to $r_2$ on the bottom plate. According to the induction hypothesis
there exist $n-2$ Hamiltonian paths in $\Q_n^{bot}-\{r_2\}$ that connect $g_1$
to $g_4$ that all begin with different edges. Similarly, there exist $n-2$
Hamiltonian paths in $\Q_n^{top}-\{r\}$ that connect $g_2$ to
$g_3$ that all begin with different edges. Let $\gamma$ be one of these paths.
Each Hamiltonian path on the bottom plate could be connected by means of the
edge $\{g_4,r_2\}$ and the bridge $\{r_2,g_3\}$ to $\gamma$. In that way, we
produce $n-2$ Hamiltonian paths of $\Q_n-\{r\}$ connecting $g_1$ to $g_2$ and
all beginning with different edges.

Now, to produce the $(n-1)$-th Hamiltonian path of $\Q_n-\{r\}$ that connects
$g_1$ to $g_2$ and begins with a different edge we proceed as follows. Produce
a Hamiltonian path of $\Q_n^{top}$ that connects $r_1$ to $g_2$. Cut this path
just before and right after $r$ and produce two bridges. Let their ends on the
bottom plate be $r_3$ and $r_4$. Then there exists a Hamiltonian path for
$\Q_n^{bot}-\{g_1\}$ that connects $r_3$ to $r_4$ ($[1,1,0,1]=2$). Then the
desired Hamiltonian path of $\Q_n-\{r\}$ that connects $g_1$ to $g_2$ is
obtained by connecting the paths constructed on the plates by means of the
bridges after removing the edges incident to $r$ from the path on the top
plate and attaching the edge $\{g_1,r_1\}$ to the resulting path.
\end{proof}

Let  $a$ be  a vertex in $\Q_n.$ There is a unique vertex $\bar{a}$ in $\Q_n$ at distance $n$ from $a.$ The coordinates
of $\bar{a}$ are the negation of the corresponding coordinates of $a.$

Let $\{r,g\} $ be a pair of a red and a green vertex in $\Q_3.$ We define the set of pairs of vertices
$\mathcal{B}_{\{r,g\}}$ in the following way: if $r = \bar{g}$ then $ \{r',g'\} \in \mathcal{B}_{\{r,g\}}$ if and only
if $ \{r', g'\} \neq \{r,g\}$ and $ r'=\bar{g'}$; if $r \neq \bar{g}$ then $\mathcal{B}_{\{r,g\}} =
\{\{\bar{r},\bar{g}\}\}.$

\begin{lemma} \label{q3minustwovertices} Let $r,g$ be a red and a green
vertex in $\Q_3,$ and let $r_1,g_1$ be a red and a green vertex in $\Q_3- \{r,g\}.$ Then
\begin{itemize}
\item[(1)] If $\{r_1, g_1\} \not \in \mathcal{B}_{\{r,g\}}$ then there exists a Hamiltonian path of $\Q_3 - \{r,g\}$
    that connects $r_1$ to $g_1.$

\item[(2)] If $\{r_1, g_1\}  \in \mathcal{B}_{\{r,g\}}$ then there does not exist a Hamiltonian path of $\Q_3 -
    \{r,g\}$ that connects $r_1$ to $g_1.$

\item[(3)] If $\{r_1, g_1\}  \in \mathcal{B}_{\{r,g\}}$ then there exist two distinct $2-$path coverings of $\Q_3
    -\{r,g\}$, with four distinct end points, with one path starting at $r_1,$ the other starting at $g_1,$ and both
    paths of length two.

\item[(4)] There exist two distinct $3-$path coverings of $\Q_3 -\{r,g\}$ with paths of length one.
\end{itemize}
\end{lemma}

\begin{proof} By inspection. \end{proof}

\begin{lemma} $([2,0,1,0]=4)$ \label{twodeletedonepath} Let $n \geq 2$ and
$r,r_1,g,g_1$ be two red and two green vertices in $\Q_n$. If $n=2$ or $n\ge 4$
then there exists a Hamiltonian path for $\Q_n-\{r_1,g_1\}$ connecting $r$ to $g.$
If $n=3$ the same conclusion follows provided $\{r,g\} \not \in
\mathcal{B}_{\{r_1,g_1\}}.$ 
\end{lemma}

\begin{proof} The statement is obvious for $n=2$ and for $n=3$ the claim is 
contained in Lemma \ref{q3minustwovertices}(1). Also, 
Lemma \ref{q3minustwovertices}(2) shows that $[2,0,1,0]\ge4$.

Now, let $n\ge 4.$ Produce two plates to separate $r$ from $r_1$ and assume that $r_1$ is on the top plate. Then $g$ and
$g_1$ can be distributed in four different ways:
\begin{itemize}
\item[(1)] both are on the top plate; \item[(2)] $g$ is on the top plate and $g_1$ is on the bottom plate;
    \item[(3)] $g_1$ is on the top plate and $g$ is on the bottom plate; and \item[(4)] both are on the bottom
    plate.
\end{itemize}

The following diagrams show how to handle these cases.

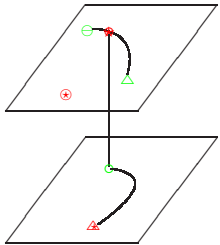
\begin{picture}(330,100)(11.5,2)

\twoplatescl

\put(20,55){\deletedred}

\put(30,5){\firsttermred}

\put(43,60){\firsttermgreen}

\put(28,79){\deletedgreen}

\thinlines \qbezier(46,64)(50,80)(33,81)

\put(37,78){\bridgered}

\put(39,79){\line(0,-1){50}}

\put(37,26){\bridgegreen}

\thinlines \qbezier(40,28)(60,25)(35,7)

\put(91,83){\parbox[t]{375pt}{(1) Use $[1,1,0,1]=2$ to produce a path covering of the top plate connecting the
green terminal $g$ to the deleted green vertex $g_1$ avoiding the deleted red vertex $r_1$. Cut this path just before
the deleted green vertex and produce a bridge from the cut vertex. Use $[0,0,1,0]=1$ to produce a Hamiltonian path of
the bottom plate that connects the lower vertex of the bridge to the red terminal $r$.}}
\end{picture}

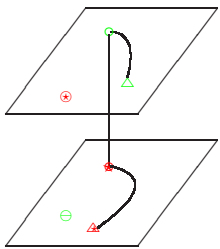
\begin{picture}(330,100)(11.5,0)

\twoplatescl

\put(20,55){\deletedred}

\put(30,5){\firsttermred}

\put(43,60){\firsttermgreen}

\put(20,10){\deletedgreen}

\thinlines \qbezier(46,64)(50,80)(40,81)

\put(37,79){\bridgegreen}

\put(39,80){\line(0,-1){50}}

\put(37,28){\bridgered}

\thinlines \qbezier(40,30)(60,25)(35,7)

\put(91,83){\parbox[t]{375pt}{(2) Find a bridge with green vertex on the top different from $g$ and red vertex
on the bottom different from $r$. Use $[1,1,0,1]=2$  to connect the green terminal to the bridge avoiding the red
deleted vertex. Use $[1,1,0,1]=2$ to produce a Hamiltonian path of the bottom plate that connects the lower vertex of
the bridge to the red terminal avoiding the deleted green vertex.}}
\end{picture}

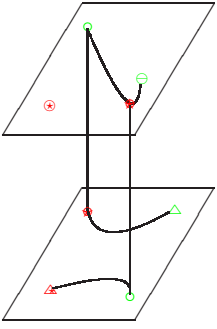
\begin{picture}(330,130)(11.5,-6)

\twoplates

\put(15,80){\deletedred}

\put(15,10){\firsttermred}

\put(62,40){\firsttermgreen}

\put(50,90){\deletedgreen}

\thinlines \qbezier(52,89)(50,68)(32,110)

\put(46,80){\bridgered}

\put(48,81){\line(0,-1){71}}

\put(46,7){\bridgegreen}

\thinlines \qbezier(48,10)(50,20)(19,12)

\put(30,109){\bridgegreen}

\put(32,110){\line(0,-1){68}}

\put(30,39){\bridgered}

\thinlines \qbezier(32,42)(35,25)(63,41)

\put(91,113){\parbox[t]{375pt}{(3) Find a bridge with green vertex on the top different from $g_1$ and red vertex
on the bottom different from $r$. Use $[1,1,0,1]=2$ and Lemma \ref{twodifferentedges}  to connect the upper vertex of
the bridge to the deleted green vertex avoiding the deleted red vertex and making sure that the vertex immediately next
to the deleted green vertex along the path is not adjacent to the green terminal on the bottom plate. Cut the path just
before the deleted green vertex and produce a bridge from the cut vertex. Use $[0,0,2,0]=2$ to produce a $2-$path
covering of the bottom plate that connects the lower vertices of the bridges to the appropriate terminals.}}
\end{picture}

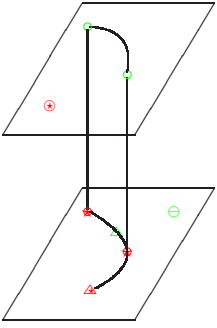
\begin{picture}(330,130)(11.5,0)

\twoplates

\put(15,80){\deletedred}

\put(30,10){\firsttermred}

\put(40,32){\firsttermgreen}

\put(62,40){\deletedgreen}

\put(45,24){\bridgered}

\put(47,26){\line(0,1){66}}

\put(45,91){\bridgegreen}

\thinlines \qbezier(47,94)(50,110)(33,111)

\put(30,109){\bridgegreen}

\put(32,110){\line(0,-1){68}}

\put(30,39){\bridgered}

\thinlines \qbezier(33,41)(60,25)(35,12)

\put(91,113){\parbox[t]{375pt}{(4) Find a bridge with a red vertex on the bottom plate different from $r$. Use
$[1,1,0,1]=2$  to connect the red terminal on the bottom  plate to the lower vertex of the bridge avoiding the green
deleted vertex. This path must pass through the green terminal. Cut the path just before the green terminal and produce
another bridge at the cut vertex. On the top plate use $[1,1,0,1]=2$ to connect the upper vertices of the bridges
avoiding the red deleted vertex.}}
\end{picture}
\vskip-12pt
\end{proof}

\begin{corollary}\label{corpathconnectingadjacentvertices}
Let $n \ge 4$ and $\F$ be any neutral fault of mass $2$ in $\Q_n.$ Then for any 
edge $e$ in $\Q_n - \F$ there exists a
Hamiltonian cycle of $\Q_n - \F$ that contains $e$.
\end{corollary}

\begin{lemma}\label{locke2}$([4]=4)$ Let $n \ge 4$ and $\F$ be any neutral fault 
of mass $4$ in $\Q_n.$ Then $\Q_n - \F$ is Hamiltonian. The claim is not true for 
$n=3$.
\end{lemma}

\begin{proof} 
Since $[2k]\ge k+2$ for each integer $k\ge 0$, we have $[4]\ge 4$.
 
Let $n \ge 4$, $r_1,r_2$ be the two red, and $g_1,g_2$ be the two
green vertices in $\F.$ Split $\Q_n$ into two plates with $r_1$ on  the top plate 
and $r_2$ on the bottom plate. There are two essentially different cases that 
depend on the distribution of the green deleted vertices between
the plates.

{\emph{Case $1.$}} The two deleted green vertices are on the top plate.

Use $[1,1,0,1]=2$ to produce a  path on the top plate that connects the two deleted green vertices and visits all the
vertices of the top plate except the deleted red vertex. From the vertices immediately next to the deleted green
vertices along the constructed path, produce bridges to connect to the bottom plate. Use $[1,1,0,1]=2$ to connect the
lower vertices of these bridges by a path on the bottom plate that visits all the vertices of the bottom plate except
the deleted red vertex. To produce the desired Hamiltonian cycle in $\Q_n - \F$ remove from the path constructed on the
top plate the edges connecting to the deleted green vertices and attach to the resulting path, by means of the bridges,
the path constructed on the bottom plate.

{\emph{Case $2.$}} $g_1$ is on the top plate and $g_2$ is on the bottom plate.

We produce a Hamiltonian cycle of $\Q_n^{top} -\{r_1,g_1\}$ using $[2]=3$. Along this cycle find two consecutive
vertices $r_3, g_3$ with adjacent vertices on the bottom plate $g_4$ and $r_4$, respectively, with $g_4\ne g_2$ and
$r_4\ne r_2$, and such that $g_4$ is adjacent to $r_2.$ This last requirement is important for $n=4$ but irrelevant for
higher dimensions. It guarantees that $\{r_4,g_4\} \not \in \mathcal{B}_{\{r_2,g_2\}}$ when the bottom plate is
isomorphic to $\Q_3.$ (To see that such vertices $r_3$ and $g_3$ exist just take $g_4$ to be a neighbor of $r_2$ in
$\Q_n^{bot}- \{g_2\}$ which is not a neighbor of $r_1$ (since $n\ge 4$ such a neighbor exists). Then denote by $r_3$ the
neighbor of $g_4$ in $\Q_n^{top}$. Clearly $r_3$ will be different from $r_1$ and will belong to the Hamiltonian cycle
on the top. Now take $g_3$ to be a neighbor of $r_3$ in that cycle which is not a neighbor of $r_2$.) Then using
$[2,0,1,0]=4$ we can produce a Hamiltonian path of $\Q_n^{bot}- \{r_2, g_2\}$ that connects $r_4$ to $g_4.$ The desired
Hamiltonian cycle of $ \Q_n- \F$ is formed by connecting the  path on the bottom plate to the  cycle on the top plate by
mean of the bridges $\{r_3,g_4\}, \{r_4, g_3\}$ and, of course, removing the edge $ \{r_3, g_3\}.$
\end{proof}

For the sake of brevity, from now on, we  adopt the following conventions for the proofs using diagrams. The paths drawn
on each plate are assumed to form path coverings of that plate so we indicate in the diagram just what vertices are
connected by these paths. From the diagram it will be clear which vertices are avoided by the path covering. A sentence
such as ``we find a bridge with green at the top" means that we select a green vertex on the top plate such that neither
it nor its adjacent vertex on the bottom plate is a terminal or a deleted vertex, and we produce the bridge between
these two vertices. A sentence such as ``we choose two adjacent bridges along this path to do surgery" means that 1) we
select two consecutive vertices along the mentioned path such that neither them nor their adjacent vertices on the other
plate are terminals or deleted vertices; 2) we produce bridges from the selected vertices to the other plate; and 3) we
remove the edge that connects the selected vertices. At the end of each construction, when we produce the final path
covering, all the edges of the original path covering that were connected to deleted vertices, if such edges exist, must
be cut out. The desired path covering is formed by the paths that connect figures of the same color and shape to each
other. These paths should be clear to the reader from the diagrams.

The following lemma was independently obtained by Caha and Koubek
\cite[Corollary 10]{cahakoubek2}. However their proof is too
involved. We provide here a simpler and direct proof.

\begin{lemma}\label{twochargedpaths} $([0,0,0,2]=4)$ Let $n \geq 3$ and
$r,r_1,g,g_1$ be two red and two green vertices in $\Q_n$. If $n\ge 4$ then there 
exists a $2-$path covering of $\Q_n$ with one path connecting $r$ to
$r_1$ and the other connecting $g$ to $g_1$.
If $n=3$ the same conclusion holds provided that $r$ and $r_1$ are contained in a 
two dimensional subcube $\alpha$ of $\Q_3$ and exactly one of the
vertices $g$ or $g_1$ is contained in $\alpha$. 
\end{lemma}

\begin{proof} The claim is straightforward for $n=3.$ Also, one can
directly verify that if $r$ and $r_1$ are contained in a two dimensional subcube $\alpha$ of $\Q_3$ and none or both of
the vertices $g$ and $g_1$ are contained in $\alpha$ then there does not exist a $2-$path covering of $\Q_3$ with one
path connecting $r$ to $r_1$ and the other connecting $g$ to $g_1$. Therefore $[0,0,0,2]\ge 4.$

Let $n\ge 4.$ Split $\Q_n$ into two plates that separate the two red terminals. We can assume that $r \in \Q_n^{top}$
and $ r_1\in \Q_n^{bot}.$ There are two essentially different cases that depend on the distribution of the green
terminals between the plates: (1) the two green terminals are on the top plate; and (2) $g$ is on the top plate and
$g_1$ is on the bottom plate. These cases can be handled as explained in the following diagrams.

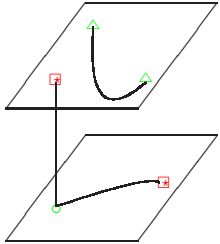
\begin{picture}(330,100)(11.5,0)

\twoplatescl

\put(30,80){\firsttermgreen}

\put(57,20){\secondtermred}

\thinlines \qbezier(33,81)(30,40)(53,60)

\put(50,60){\firsttermgreen}

\put(16,59){\secondtermred}

\put(17,10){\bridgegreen}

\thinlines \qbezier(58,22)(55,25)(19,13)

\put(19,60){\line(0,-1){47}}

\put(91,83){\parbox[t]{375pt}{(1) Use $[1,1,0,1]=2$ to find a Hamiltonian path of $\Q_n^{top}- \{r\} $ that
connects $g$ to $g_1.$ Connect the top red terminal $r$ to the bottom plate by a bridge. Use $[0,0,1,0]=1$ to find a
Hamiltonian path of the bottom plate that connects the lower vertex of the bridge to the red terminal $r_1$ on the
bottom plate.}}

\end{picture}

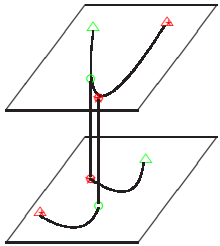
\begin{picture}(330,100)(11.5,0)

\twoplatescl

\put(30,80){\firsttermgreen}

\put(10,10){\firsttermred}

\thinlines \qbezier(33,80)(27,30)(60,82)

\put(50,30){\firsttermgreen}

\put(58,82){\firsttermred}

\put(30,60){\bridgegreen}

\put(30,22){\bridgered}

\put(33,53){\bridgered}

\put(33,12){\bridgegreen}

\put(32,61){\line(0,-1){36}}

\put(35,54){\line(0,-1){39}}

\thinlines \qbezier(32,24)(50,12)(52,30)

\thinlines \qbezier(35,13)(30,0)(13,10)

\put(91,83){\parbox[t]{375pt}{(2) Use $[0,0,1,0]=1$ to produce a Hamiltonian path of the top plate that connects
the two terminals. While traversing the path starting from the green terminal, find an edge  whose first vertex is green
and such that the adjacent vertices on the bottom are not terminals. Produce
 bridges from the vertices of this edge. Use $[0,0,2,0]=2$ to
produce a $2-$path covering of the bottom plate that connects the lower vertices of the bridges to the appropriate
terminals.}}
\end{picture}
\end{proof}

The following lemma is a refinement of Lemma \ref{twochargedpaths}. It shows
that one can choose which one of the two pairs
of terminals to be connected by the longer path.

\begin{lemma} \label{twochargedpathsonelong}Let $n\ge 3,$ $r,r_1$ be two
distinct red vertices and $g, g_1$ be two distinct green vertices in $\Q_n.$ If $n=3$ we also require that if $r$ and
$r_1$ are contained in a two dimensional subcube $\alpha$ of $\Q_3,$ then exactly one of the vertices $g$ or $g_1$ is
contained in $\alpha$. Then there exists a $2-$path covering of $\Q_n$ with the first path of length at least $2^{n-1}$
connecting $r$ to $r_1$ and the second path connecting $g$ to $g_1.$
\end{lemma}

\begin{proof} If $n=3$ then our claim can be verified directly.

For $n \ge 4$ we produce two plates as in the proof of Lemma \ref{twochargedpaths} and consider the same two cases. The
proof of case (1) does not need to be modified. For case (2) we assume without loss of generality that $r, g$ are on the
top plate and $r_1,g_1$ are on the bottom plate. There are three subcases to consider.

\emph{Subcase 2(a).} $g$ is not adjacent to $r_1.$

Let $r_2$ be any red vertex on the top plate that is  adjacent to vertex $g_2$ of the bottom plate different from $g_1.$
Use $[1,1,0,1]=2$ to produce a Hamiltonian path of $\Q_n^{top}-\{g\}$ that connects $r$ to $r_2.$ Let $r_3$ be the
vertex of the bottom plate that is adjacent to $g.$ Use $[0,0,2,0]=2$ to produce a $2-$path covering of the bottom plate
that connects $r_3$ to $g_1$ and $g_2$ to $r_1.$ The desired $2-$path covering of $\Q_n$ is obtained by connecting the
path produced on the plates by means of the bridges $\{r_2,g_2\}$ and $\{g,r_3\}.$

\emph{Subcase 2(b).} $r$ is not adjacent to $g_1.$

Let $r_2$ be the vertex of the top plate that is adjacent to $g_1.$ Let $g_2$ be any green vertex on the top plate
different from $g$ and adjacent to a vertex $r_3 \neq r_1$ of the bottom plate. Use $[0,0,2,0]=2$ to produce a $2-$path
covering of $\Q_n^{top}$ that connects $g$ to $r_2$ and $r$ to $g_2.$ Use $[1,1,0,1]=2$ to produce a Hamiltonian path of
$\Q_n^{bot}-\{g_1\}$ that connects $r_3$ to $r_1.$ The desired $2-$path covering of $\Q_n$ is obtained by attaching to
the paths constructed on the plates the bridges $\{r_2, g_1\}$ and $\{g_2,r_3\}.$

\emph{Subcase 2(c).} $r$ is adjacent to $g_1$ and $r_1$ is adjacent to $g.$

The care in choice of vertices below is important for dimension $n=4$ but can be relaxed for $n\ge 5.$

Let $r_2$ be any vertex of the bottom plate that is adjacent to $g_1,$ different from $r_1$, and let $g_2$ be the vertex
on the top plate that is adjacent to $r_2.$ On the top plate we can find a vertex $r_3$ whose adjacent vertex $g_3$ on
the bottom plate satisfies the following conditions: 1) $g_3$ is adjacent to $r_2;$ 2) $g_3 \neq g_1;$ and, in the case
$n=4$, we also require 3) the two-dimensional subcube that contains $r$ and $r_3$ contains exactly one of the vertices
$g$ or $g_2.$ Conditions 1) and 2) guarantee the existence of a Hamiltonian path of $\Q_n^{bot}-\{r_2,g_1\}$ that
connects $g_3$ to $r_1.$ Condition 3) guarantees the existence of a $2-$path covering of $\Q_n^{top}$ that connects $g$
to $g_2$ and $r$ to $r_3.$ The desired $2-$path covering of $\Q_n$ is obtained by attaching to the paths constructed on
the plates the bridges $\{r_2, g_2\},\{g_3,r_3\}$ and the edge $\{r_2,g_1\}.$
\end{proof}

\begin{lemma}$([1,1,1,1]=4)$ Let $n \geq 4$, $r$ be a deleted red
vertex in $\Q_n$, and $r_1,g,g_1,g_2$ be one red and three distinct green vertices in $\Q_n-\{r\}$. Then there exists a
$2-$path covering of $\Q_n-\{r\}$ with one path connecting $r_1$ to $g$ and the other connecting $g_1$ to $g_2$. The
claim is not true for $n = 3$.
\end{lemma}

\begin{proof} The following counterexample shows that $[1,1,1,1]>3$:
$n=3$, $r=(1,0,1)$, $r_1=(1,1,0)$, $g=(1,1,1)$, $g_1=(0,1,0)$, $g_2=(0,0,1)$.

Now let $n\ge 4.$ Produce two plates to separate the two green terminals $g_1$ and $g_2$ of the charged path and assume
that the deleted red $r$ and $g_1$ are on the top plate. The terminals of the neutral path $r_1$ and $g$ could be
distributed in four possible ways:
\begin{itemize}
\item[(1)] both are on the top plate; \item[(2)] the red is on the top plate and the green is on bottom plate;
    \item[(3)] the green is on the top plate and the red is on the bottom plate; \item[(4)] both are on the bottom
    plate.
\end{itemize}

The four cases can be approached as explained in the following diagrams.

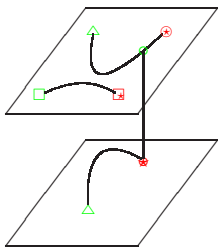
\begin{picture}(330,100)(11.5,0)

\twoplatescl

\put(30,80){\firsttermgreen}

\thinlines \qbezier(33,80)(27,50)(59,80)

\put(28,12){\firsttermgreen}

\put(58,80){\deletedred}

\put(10,55){\secondtermgreen}

\put(40,55){\secondtermred}

\put(50,72){\bridgegreen}

\put(50,30){\bridgered}

\put(52,73){\line(0,-1){40}}

\thinlines \qbezier(31,16)(30,45)(51,33)

\thinlines \qbezier(15,58)(30,65)(41,58)

\put(91,83){\parbox[t]{375pt}{(1) Use $[0,0,2,0]=2$ to produce a $2-$path covering of the top plate that
connects the two terminals of the neutral path to each other, and the green terminal of the charged path to the deleted
red. Cut the last path just before the deleted red and produce a bridge. Use $[0,0,1,0]=1$ to find a Hamiltonian path
for the bottom plate connecting the lower vertex of the bridge to the green terminal on the bottom plate.}}
\end{picture}

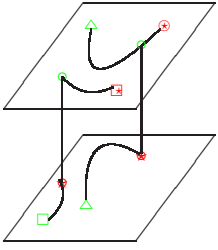
\begin{picture}(330,100)(11.5,0)

\twoplatescl

\put(30,80){\firsttermgreen}

\thinlines \qbezier(33,80)(27,50)(59,80)

\put(28,12){\firsttermgreen}

\put(20,60){\bridgegreen}

\put(20,20){\bridgered}

\put(22,62){\line(0,-1){40}}

\put(58,80){\deletedred}

\put(12,6){\secondtermgreen}

\put(40,55){\secondtermred}

\put(50,72){\bridgegreen}

\put(50,30){\bridgered}

\put(52,73){\line(0,-1){40}}

\thinlines \qbezier(21,23)(25,15)(17,8)

\thinlines \qbezier(31,16)(30,45)(51,33)

\thinlines \qbezier(22,62)(30,53)(41,58)

\put(91,83){\parbox[t]{375pt}{(2) Find a bridge with green on the top. Use $[0,0,2,0]=2$ to produce a $2-$path
covering of the top plate that connects the red terminal of the neutral path to the upper vertex of the bridge and the
green terminal of the charged path to the deleted red vertex. Cut the second  path just before the deleted red vertex
and produce a second bridge there. Use $[0,0,2,0]=2$ to produce a $2-$path covering of the bottom plate that connects
the lower vertices of the bridges to the appropriate terminals. }}
\end{picture}

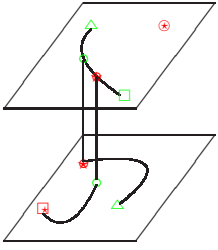
\begin{picture}(330,110)(11.5,0)

\twoplatescl

\put(30,80){\firsttermgreen}

\thinlines \qbezier(33,80)(22,70)(44,55)

\put(40,12){\firsttermgreen}

\put(33,60){\bridgered}

\put(33,20){\bridgegreen}

\put(35,62){\line(0,-1){39}}

\put(58,80){\deletedred}

\put(12,10){\secondtermred}

\put(43,53){\secondtermgreen}

\put(28,67){\bridgegreen}

\put(28,27){\bridgered}

\put(30,68){\line(0,-1){38}}

\thinlines \qbezier(30,30)(70,35)(44,14)

\thinlines \qbezier(15,10)(25,0)(35,21)

\put(91,83){\parbox[t]{375pt}{(3) Use $[1,1,0,1]=2$ to produce a Hamiltonian path of $\Q_n^{top} - \{r\}$ that
connects $g$ to $g_1.$ Traversing this path from $g$ to $g_1$  find two consecutive vertices that are not neighbors to
the green and red terminals on the bottom plate and such that the first vertex is green. Such pair of consecutive
vertices exist since the length of the path is at least six, hence there are at least three such pairs on the top and
only two vertices to avoid on the bottom. Produce bridges from these vertices. Use $[0,0,2,0]=2$ to produce a $2-$path
covering of the bottom plate that connects the lower vertices of the bridges to the appropriate terminals. }}
\end{picture}

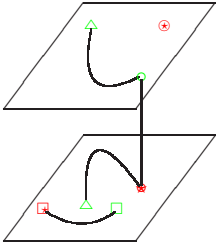
\begin{picture}(330,123)(11.5,0)

\twoplatescl

\put(30,80){\firsttermgreen}

\thinlines \qbezier(33,80)(27,50)(51,62)

\put(28,12){\firsttermgreen}

\put(58,80){\deletedred}

\put(12,10){\secondtermred}

\put(40,10){\secondtermgreen}

\put(50,60){\bridgegreen}

\put(50,18){\bridgered}

\put(52,61){\line(0,-1){40}}

\thinlines \qbezier(31,16)(30,50)(51,21)

\thinlines \qbezier(16,11)(30,2)(42,11)

\put(91,83){\parbox[t]{375pt}{(4) Find a bridge with green on the top. Use $[1,1,0,1]=2$ to find a Hamiltonian
path of the top plate connecting the green terminal of the charged path to the bridge avoiding the deleted red
vertex. Use $[0,0,2,0]=2$ to find a $2-$path covering of the bottom plate  connecting the lower vertex of the
bridge to the green terminal of the charged path and the two terminals of the neutral path.}}
\end{picture}
\vskip-12pt
\end{proof}

\begin{lemma} \label{chargeoneoneedgeonepath}
Let $n \ge 4,$ $r_1$ and $r_2 $ be two distinct red vertices in $\Q_n$ and $g$ be a green vertex that is deleted from
$\Q_n.$ Assume further that $e=\{a,b\}$ is any edge in $\Q_n - \{g\}$. Then there exists a Hamiltonian path in $\Q_n -
\{g\}$ that connects $r_1$ to $r_2$ and passes through the edge $e.$ In the case when $\{a,b\}\cap\{r_1,r_2\}=\emptyset$
we can find an oriented Hamiltonian path in $\Q_n - \{g\}$ connecting $r_1$ to $r_2$ such that the path visits the
vertex $a$ first.
\end{lemma}

\begin{proof} If the prescribed edge $e$ is not incident to any of the
prescribed end vertices $r_1, r_2,$  use $[1,1,1,1]=4$ to connect $r_1$ to $a$ and $r_2$ to $b.$ The desired (oriented)
Hamiltonian path in $\Q_n - \{g\}$ is obtained by connecting these two paths to each other through the edge $e.$

Let the prescribed edge be incident to one of the prescribed end vertices. We can assume without loss of generality that
$a = r_1.$ Then use $[2,0,1,0]=4$ to produce a Hamiltonian path in $\Q_n - \{r_1,g\}$  that connects $r_2$ to $b.$ Then
attach the edge $e$ to this path to obtain the desired Hamiltonian path in $\Q_n - \{g\}.$
\end{proof}

\begin{lemma} $([3,1,0,1]=4)$ Let $n \geq 4$ and
$g$, $r$ and $r_1$ be one green and two distinct red vertices in
$\Q_n$. Let also $g_1$ and $g_2$ be two distinct green terminals
in $\Q_n-\{g,r,r_1\}$. Then there exists a Hamiltonian path for
$\Q_n-\{g,r,r_1\}$ connecting $g_1$ to $g_2$. The claim is not
true for $n = 3$.
\end{lemma}

\begin{proof} The following counterexample shows that $[3,1,0,1]>3$:
$n=3$, $r=(1,0,1)$, $r_1=(1,1,0)$, $g=(1,1,1)$, $g_1=(0,1,0)$, $g_2=(0,0,1)$.

Now, let $n \ge 4$. There exist two plates that separate the deleted red vertices $r$ and $r_1$ and we assume that the
top plate is the one that contains the deleted green vertex $g$. We consider the three essentially different cases that
depend on the distribution of the green terminals $g_1$ and $g_2$ on the plates.

\emph{Case 1.} The two green terminals are on the top plate.

Use $[1,1,0,1]=2$ to produce a path that visits all the vertices of the top plate except the red deleted vertex and
starts at one green terminal and ends at the deleted green vertex. This path must pass through the second green
terminal. Cut this path at the vertex immediately preceding the second green terminal and at the vertex immediately
preceding the deleted green vertex along the path. From the cut vertices produce two bridges. The lower vertices of
these bridges are green. Connect them by a path on the bottom plate that visits all the vertices except the deleted red
vertex. This finishes the construction of the desired path for this case.

\emph{Case 2.} One green terminal is on the top plate and the other one is on the bottom plate.

Use $[1,1,0,1]=2$ to produce a path on the top plate that visits all the vertices except the deleted red vertex and that
starts at the green terminal and ends at the deleted green vertex. By Lemma \ref{twodifferentedges} this path can be
chosen in such a way that the vertex just before the deleted green is not adjacent to the green terminal on the bottom.
Cut the path just before the deleted green and produce a bridge from the cut vertex. Use $[1,1,0,1]=2$ to produce a path
on the bottom plate that connects the lower vertex of the bridge to the green terminal and that visits all the vertices
of the bottom plate except the red deleted vertex.

\emph{Case 3.} The two green terminals are on the bottom plate.

Use $[2]=3$ to produce a cycle on the top plate that visits all the vertices except the deleted ones.

If $n=4$, use $[1,1,0,1]=2$ to produce a path on the bottom plate that visits all the vertices except the deleted red
vertex and has the two green terminals as end vertices. At least one non-terminal vertex $u$ of this path is adjacent to
a vertex $v$ in the cycle on the top plate. Since the degree of each of these vertices relative to its plate is three,
one of the neighbors of $u$ in the bottom path must be adjacent to one of the neighbors of $v$ in the cycle produced on
the top plate. In other words, there exist two parallel bridges such that the edges connecting their ends on the bottom
and on the top plate belong to the path on the bottom plate and to the cycle on the top plate, respectively. Use these
bridges to do surgery to connect the bottom path to the cycle on the top plate by means of the bridges. This finishes
the construction of the desired path for this case when $n=4.$

If $n\geq 5$ then the plates are of dimension greater than three. Thus, there exist two consecutive vertices along the
cycle constructed on the top plate such that their adjacent vertices on the bottom plate are neither deleted vertices
nor terminal vertices. Select two such vertices and cut the cycle there and produce bridges to the bottom plate. Then
use Lemma \ref{chargeoneoneedgeonepath} to produce a path on the bottom plate that 1) starts at one green terminal and
ends at the other green terminal; 2) visits all the vertices of the bottom plate except the deleted red vertex; and 3)
passes through the edge incident to the lower vertices of the two bridges. Finally, do surgery to connect the path on
the bottom plate to the cycle on the top plate through the bridges. The result is the desired path. This finishes the
construction of the desired path for this case when $n \ge 5.$
\end{proof}

\begin{lemma}\label{new} Let $n \geq 4$ and
$g$ and $r$ be a green and a red vertex in $\Q_n$. Let also $g_1$ and $g_2$ be two distinct green vertices in
$\Q_n-\{g,r\}$. Then there exists a Hamiltonian cycle for $\Q_n-\{g,r\}$ such that the shortest distance between $g_1$
and $g_2$ along that cycle is at least four.
\end{lemma}

\begin{proof}
Split $\Q_n$ into two plates such that $g_1$ is on the top plate and $g_2$ is on the bottom plate. There are two cases
to consider.

\emph{Case 1.} $r$ and $g$ are on the top plate.

Use $[2]=3$ to find a Hamiltonian cycle for $\Q_n^{top}-\{r,g\}$.
Choose an edge $(g_3,r_3)$ from this cycle such that $g_1\ne g_3$
and $r_3$ is not adjacent to $g_2$. Cut the cycle at that edge and
connect the resulting path with bridges to the bottom plate. Use
$[0,0,1,0]=1$ to find a Hamiltonian path for the bottom plate that
connects the bottom vertices of the two bridges. The resulting
Hamiltonian cycle of $\Q_n-\{g,r\}$ has the required property.

\emph{Case 2.} $r$ is on the top plate, $g$ is on the bottom plate.

Find two bridges with green vertices on the top plate that avoid
$g_1$. Use $[1,1,0,1]=2$ to find Hamiltonian paths for
$\Q_n^{top}-\{r\}$ and $\Q_n^{bot}-\{g\}$, respectively, that
connect the end vertices of the bridges. The resulting Hamiltonian
cycle of $\Q_n-\{g,r\}$ has the required property.
\end{proof}

\section{Larger faults and sets of prescribed ends}

In this section we identify the hypercube $\Q_n$ with the group $ {\bf Z}_2 ^n .$ We view $\Q_n$ as a Cayley graph with
the standard system of generators  $ {\bf S} = \{ e_1 = (1,0,\dots,0), e_2=(0,1,0,\dots,0),  \dots,
e_n=(0,\dots,0,1)\}.$ An oriented edge in $\Q_n$ is represented by $(a,x),$ where $a$ is the starting vertex and $x$ is
an element from the system of generators $\bf S.$  A path is represented by $(a,\omega),$ where $a$ is the initial
vertex and $\omega$ is a word with letters from ${\bf S}.$ If $\omega = x_1,x_2,\dots,x_k$ then the path $(a,\omega) $
is the path $a, ax_1, ax_1x_2, \dots, ax_1x_2\cdots x_n.$ The algebraic content of a word $\omega$ is the element of
${\bf Z}_2 ^n$ that is obtained by multiplying all the letters of $\omega. $ A path $(a,\omega)$ is simple if no subword
of $\omega$ is algebraically equivalent to the identity $(0,0,\dots,0).$ A path $(a, \omega)$ is a cycle if $\omega$ is
algebraically equivalent to the identity but no proper subword of $\omega$ is algebraically equivalent to the identity.

We shall use the following notation: $\omega^R $ means the reverse word of $\omega; $ $\omega'$ denotes the word
obtained after the last letter is deleted from $\omega;$ $ \omega^*$ is the word obtained after the first letter is
deleted from $\omega;$ $ \varphi(\omega) $ is the first letter of $\omega,$ and $ \lambda(\omega)$ is the last letter of
$\omega.$  The letter $v$ shall be reserved for steps connecting two plates. The letters $x,y,...$ shall be reserved to
represent steps along the plates.

The following lemma can be proved by inspection.

\begin{lemma} \label{twoorientedpaths} Let $r, r_1, r_2$ be three distinct
red vertices and $g, g_1, g_2$ be three distinct green vertices in $\Q_3.$ Then there exist two oriented paths
$\gamma_1, \gamma_2$  such that
\begin{itemize}
\item[(i)] $\gamma_1$ is Hamiltonian in $\Q_3 - \{g\}$ and connects $r_1$ to $r_2$; \item[(ii)] $\gamma_2$ is
    Hamiltonian in $\Q_3 - \{r\} $ and connects $g_1$ to $g_2$; and \item[(iii)] $\gamma_1$ and $\gamma_2$ share an
    edge that is traversed in the same direction in both paths.
\end{itemize}
\end{lemma}

The following lemma is a generalization of Lemma \ref{twoorientedpaths}.

\begin{lemma} \label{twoorientedpathsn} Let $n \geq 4$ and $r_1$, $r_2$,
$g_1$, $g_2$, $g_3$, $g_4$ be two distinct red and four distinct green vertices in $\Q_n$ such that $r_1,g_1,g_2 \in
\Q_n^{top}$ and $r_2,g_3,g_4 \in \Q_n^{bot}$. Then there exist two oriented paths $\gamma_1, \gamma_2$ such that
\begin{itemize}
\item[(i)] $\gamma_1$ is Hamiltonian in $\Q_n^{top} - \{r_1\}$ and connects $g_1$ to $g_2$; \item[(ii)] $\gamma_2$
    is  Hamiltonian in $\Q_n^{bot} - \{r_2\}$ and connects $g_3$ to $g_4$; and \item[(iii)] there exist an edge
    $(a,a x) \in \gamma_1$ such that $(a v,a v x) \in \gamma_2$ and both edges are traversed in the same direction
    in both paths.
\end{itemize}
\end{lemma}

\begin{proof}
The proof is by induction. If $n=4$ then the claim is contained in Lemma \ref{twoorientedpaths}. If $n > 4$ then choose
an edge $(a,a x) \in \Q_n^{top}$ such that none of the given vertices $r_1$, $r_2$, $g_1$, $g_2$, $g_3$, $g_4$ is
incident to $(a,a x)$ or $(a v,a v x)$ and apply Lemma \ref{chargeoneoneedgeonepath} to construct $\gamma_1$ and
$\gamma_2$ in the desired way.
\end{proof}

\begin{lemma} $([2,2,0,2]=4)$ Let $n \geq 4,$ $\F=\{r_1,r_2\}$ be a fault
with two distinct red vertices and $g_1$, $g_2$, $g_3$, $g_4$ be
four distinct green vertices in $\Q_n$. Then there exists a
$2-$path covering of $\Q_n-\F$ with one path connecting $g_1$ to
$g_2$ and the other connecting $g_3$ to $g_4$. The claim is not
true for $n = 3$.
\end{lemma}

\begin{proof} The following counterexample shows that $[2,2,0,2]>3$:
$n=3$, $r_1=(1,1,0)$, $r_2=(1,0,1)$, $g_1=(0,1,0)$, $g_2=(0,0,1)$, $g_3=(1,0,0)$, $g_4=(1,1,1)$.

Now let $n \ge 4.$ Split the hypercube in such a way that $r_1$ is on the top plate and $r_2$ is on the bottom plate.
Then consider four cases that depend on the distribution of the green terminals on the plates.

\emph{Case 1.} All green terminals $g_1$, $g_2$, $g_3$, $g_4$ are on the top plate.

Use $[1,1,0,1]=2$ to find a Hamiltonian path $(g_1, \omega)$ of $\Q_n^{top} - \{r_1\} $ that connects $g_1$ to $g_2.$
Let $\omega = \xi \eta \theta $ with $g_1 \xi = g_3, g_3 \eta = g_4,$ and $g_4 \theta = g_2,$ where $g_3, g_4$ are
renumbered, if necessary.

Use $[1,1,0,1]=2$ to find a Hamiltonian path $(g_1 \xi' v, \mu)$ of $\Q_n^{bot} - \{r_2\}$ that connects $g_1 \xi' v$ to
$g_1 \xi \eta \varphi(\theta) v.$ Then the desired $2-$path covering of $\Q_n - \F$ is $(g_1, \xi' v \mu v \theta^*),$
$(g_3, \eta).$

\emph{Case 2.} $g_1$, $g_2$, $g_3$ are on the top and $g_4$ is on the bottom plate.

Use $[1,1,0,1]=2$ to find a Hamiltonian path $(g_1, \omega)$ of $\Q_n^{top} - \{r_1\} $ that connects $g_1$ to $g_3.$
Let $\omega = \xi \eta $, where $g_1 \xi = g_2$ and $g_2 \eta = g_3.$

\emph{Subcase 2(a).} $g_2\varphi(\eta) v \neq g_4$.

On the bottom plate use again $[1,1,0,1]=2$ to find a Hamiltonian path $(g_2 \varphi(\eta) v, \mu)$ of $ \Q_n^{bot}-
\{r_2\}$ that connects $g_2\varphi(\eta) v$ to $g_4$. Then the desired $2-$path covering of $\Q_n - \F$ is $(g_1, \xi)$,
$(g_3, (\eta^R)'v \mu).$

\emph{Subcase 2(b).} $g_2\varphi(\eta) v = g_4$.

Either $g_1$ or $g_2$ is not adjacent to $r_2$. Without loss of generality assume that it is $g_1$. If $n\ge 5,$ use
$[2,0,1,0]=4$ to find a Hamiltonian path $(g_1 v, \mu)$ of $ \Q_n^{bot}- \{r_2,g_4\}$ that connects $g_1 v$ to $g_1
\varphi(\xi) v$. Then the desired $2-$path covering of $\Q_n - \F$ is $(g_1, v \mu v \xi^*)$, $(g_3, (\eta^R)'v).$

The same argument works for $n=4$ whenever $\{g_1 v, g_1 \varphi(\xi)v\} \not \in \mathcal{B}_{\{g_4, r_2\}}$ (Lemma \ref{twodeletedonepath}). If $\{g_1 v, g_1
\varphi(\xi)v \}  \in \mathcal{B}_{\{g_4, r_2\}}$ then  the distance from $g_1 \varphi(\xi)v$ to $r_2$ is three and
therefore $g_1 \varphi(\xi^*)v\neq r_2$ and $\{g_1 \varphi(\xi) v, g_1 \varphi(\xi^*)v\} \not \in \mathcal{B}_{\{g_4,
r_2\}}.$ Then use Lemma \ref{twodeletedonepath} to find a Hamiltonian path $(g_1 \varphi(\xi) v, \mu)$ of $ \Q_n^{bot}- \{r_2,g_4\}$
that connects $g_1 \varphi(\xi) v$ to $g_1 \varphi(\xi^*) v$. The desired $2-$path covering of $\Q_n - \F$ is $(g_1,
\varphi(\xi)v \mu v \xi^{**})$, $(g_3, (\eta^R)'v).$

\emph{Case 3.} $g_1$, $g_2$ are on the top and $g_3$, $g_4 $ are on the bottom plate.

Use $[1,1,0,1]=2$ to find a Hamiltonian path $(g_1, \omega)$ of $\Q_n^{top} - \{r_1\}$ that connects $g_1$ to $g_2$ and
use again $[1,1,0,1]=2$ to find a Hamiltonian path $(g_3, \mu)$ of $\Q_n^{bot} - \{r_2\}$ that connects $g_3$ to $g_4.$
Then the desired $2-$path covering of $\Q_n - \F$ is $(g_1, \omega)$, $(g_3, \mu).$

\emph{Case 4.} $g_1$, $g_3$ are on the top and $g_2$, $g_4$ are on the bottom plate.

According to Lemma \ref{twoorientedpathsn} there exist an oriented Hamiltonian path $\gamma_1 = (g_1,\xi x \eta)$ of
$\Q_n^{top}- \{r_1\}$ connecting $g_1$ to $g_3$ and an oriented Hamiltonian path $\gamma_2 = (g_2, \mu x \theta) $ of
$\Q_n^{bot}- \{r_2\}$ connecting $g_2$ to $g_4$ such that $g_1 \xi v = g_2 \mu.$ The desired $2-$path covering is $(g_1,
\xi v \mu^R)$, $(g_3, \eta^R v \theta).$
\end{proof}

In some proofs it is useful to be able to find Hamiltonian paths that pass through each element of a given set of
vertices in such a way that the distance between two consecutive elements of that set along the path is at least $4$.
The following lemma gives a situation when that can be done. It will be used in the proofs of Lemma \ref{locke3} and
Lemma \ref{locke4}.

\begin{lemma}\label{separatinglemma} Let $n \ge 3$,
$L=\{g_1, g_2, \dots, g_{n-1}\}$ be a set of green vertices and $r$ be a red vertex in $\Q_n.$ Then there exists a
Hamiltonian path in $\Q_n-\{r\}$ that connects $g_1$ to $g_{n-1}$ in such a way that the distance along the path between
any two vertices in $L$ is at least $4.$
\end{lemma}

\begin{proof} The proof is by induction. The statement is obvious for $n=3.$
Let $n \geq 3$ and $ L=\{g_1,g_2, \dots, g_{n-1},g_n\}$ be a set of $n$ green vertices and $r$ be any red vertex in
$\Q_{n+1}.$  Produce plates in a way that $g_1 \in \Q_{n+1}^{top}$ and $g_n \in \Q_{n+1}^{bot}.$ We can assume that $r
\in \Q_{n+1}^{top}$ by renumbering $g_1$ and $g_n$, if necessary.

If $ g_1 $ is the only element of $L$ in $\Q_{n+1}^{top}$ then use $[1,1,0,1]=2$ to produce a Hamiltonian path $(g_1,
\xi)$ of $\Q_{n+1}^{top}-\{r\}$ that connects $g_1$ to $g_2 x v,$ where $x$ is any letter different from $v$ and such
that $g_2 x v \neq g_1$. By the induction hypothesis there is a Hamiltonian path $(g_2, \eta)$ of
$\Q_{n+1}^{bot}-\{g_2x\}$ that connects $g_2$ to $g_n$ and such that the distance between any two different elements of
$L$ along this path is at least $4.$ The desired Hamiltonian path of $\Q_{n+1}- \{r\}$ for this case is  $ (g_1, \xi v x
\eta ).$

If in addition to $g_1$ there is another element $g_i \in L \cap \Q_{n+1}^{top}$ (the total number of such elements
cannot be more than $n-2$) then use the induction hypothesis to produce a Hamiltonian path $(g_1, \xi)$ of
$\Q_n^{top}-\{r\}$ that connects $g_1$ to $g_i$ and such that the distance between any two elements of $L$ along this
path is at least $4.$ On the bottom plate there are at most $n-2$ elements of $L.$ Therefore, there exists a letter $x$
such that $g_i vx$ is not in $L.$ By the induction hypothesis there is a Hamiltonian path $(g_ivx, \eta)$ of
$\Q_{n+1}^{bot}-\{g_iv\}$ that connects $g_ivx$ to $g_n$ and such that the distance between any two elements from $L$
along the path is at least $4.$  The desired Hamiltonian path of $\Q_{n+1}- \{r\}$ for this case is  $ (g_1, \xi v x
\eta ).$
\end{proof}

\begin{lemma}\label{locke3}$([6]=5)$ Let $ n \ge 5 $ and $\F$ be any neutral
fault of mass $6$ in $\Q_n.$ Then $\Q_n - \F$ is Hamiltonian. The claim is not 
true if $n=3$ or $n=4$.
\end{lemma}

\begin{proof} Since $[2k]\ge k+2$ for each integer $k\ge 0$, we have $[6]\ge 5$.

Let $n \ge 5$ and $ \F = \{r_1, r_2, r_3, g_1, g_2, g_3 \}$ be
such that the first three vertices are red and the last three vertices are green. 
Produce two plates in such a way that
$r_1$ and $r_2$ are on the top plate and $r_3$ is on the bottom plate. Then consider the four essentially different
cases that depend on the distribution of the deleted green vertices on the plates.

\emph{Case $1.$} The three deleted green vertices are on the top plate.

Use $[4]=4$ to find a Hamiltonian cycle $(g_3, \xi )$ of $ \Q_n^{top}- \{r_1,r_2,g_1,g_2\}.$  Then use $[1,1,0,1]=2$ to
find a Hamiltonian path $(g_3\varphi(\xi) v, \eta) $ of $\Q_n^{bot}- \{r_3\}$ that connects $g_3 \varphi(\xi) v $ to
$g_3 \xi' v.$  The desired Hamiltonian cycle of $\Q_n - \F$ for this case is $(g_3\varphi(\xi), v \eta v (\xi^R)'^*).$

\emph{Case $2.$} $g_1$ and $g_2$ are on the top plate and $g_3$ is on the bottom plate.

Use $[4]=4$ to produce a Hamiltonian cycle on $\Q_n^{top} - \{r_1,r_2,g_1,g_2\}.$ Let $a,b$ be two consecutive vertices
along this cycle whose respective adjacent vertices on the bottom plate $c,d$ are not deleted vertices. Use
$[2,0,1,0]=4$ to connect $c$ to $d$ by a Hamiltonian path of $\Q_n^{bot}-\{r_3,g_3\}.$ The desired Hamiltonian cycle of
$\Q_n - \F$ for this case is obtained by removing the edge $\{a,b\}$ from the cycle constructed on the top plate and
attaching to the resulting path by means of the bridges $\{a,c\},\{b,d\}$ the path constructed on the bottom plate.

\emph{Case $3.$} $g_1$ is on the top plate and $g_2$ and $g_3$ are on the bottom plate.

Let $g_4,g_5$ be any two green non-deleted vertices on the top plate such that their respective adjacent vertices
$r_4,r_5$ on the bottom plate are also non-deleted. Use $[3,1,0,1]=4$ to produce a Hamiltonian path of $\Q_n^{top} -
\{r_1,r_2,g_1\}$ that connects $g_4$ to $g_5.$ In the same way produce a Hamiltonian path of $\Q_n^{bot} -
\{r_3,g_2,g_3\}$ that connects $r_4$ to $r_5.$ The desired Hamiltonian cycle of $\Q_n- \F$ for this case is obtained by
attaching the resulting paths to each other by means of the bridges $\{g_4,r_4\}, \{g_5,r_5\}.$

\emph{Case $4.$} The three green deleted vertices are on the bottom.

Use Lemma \ref{separatinglemma} to find a Hamiltonian path $(g_1, \xi)$ of $\Q_n^{bot} - \{r_3\}$ that connects $g_1$ to
$g_3$ and such that $\xi = \eta \theta,$ with $g_1 \eta = g_2,$ and both $\eta $ and $\theta $ have length at least
four. Then use $[2,2,0,2]=4$ to produce a $2-$path covering of $\Q_n^{top}-\{r_1,r_2\}$ with paths $(g_1 \varphi(\eta)v,
\mu ), (g_1 \eta' v, \nu )$ connecting $g_1 \varphi(\eta) v $ to $g_2 \varphi(\theta)v$ and $g_1 \eta' v $ to $g_2
\theta' v $, respectively. The desired Hamiltonian cycle of $\Q_n - \F$ for this case is $(g_1\varphi(\eta), v \mu v
\theta'^* v \nu^R v (\eta^R)'^*).$
\end{proof}

\begin{lemma}$([4,0,1,0]=5)$ Let $n \geq 5$, $r, r_1, r_2$ be three
distinct red vertices and $g, g_1, g_2$ be three distinct green
vertices in $\Q_n.$ Then there exists a Hamiltonian path of $\Q_n
- \{r_1,r_2,g_1,g_2\}$ that connects $r$ to $g.$ The claim is not
true if $n=3$ or $n=4$.
\end{lemma}

\begin{proof}
Let $r = (0,1,0,0)$, $r_1 = (1,0,0,0)$, $r_2 = (1,1,1,0)$ and $g =(1,0,0,1),$ $g_1 = (1,1,1,1)$, $g_2 = (0,0,1,1)$ be
vertices in $\Q_4$. Then one can verify directly that a Hamiltonian path of $\Q_4 - \{r_1,r_2,g_1,g_2\}$ connecting $r$
to $g$ does not exist.

Let $n \ge 5$. Choose two plates that separate the deleted red vertices and consider the six essentially different cases
depending on the distribution of the green deleted vertices and the terminals on the plates. We can assume that $r_1$ is
the deleted red vertex on the top plate and $r_2$ is the deleted red vertex on the bottom plate.

\emph{Case A.} The two deleted green vertices are on the top plate.

\emph{Subcase A1.} The two terminals are on the top plate.

Use $[2,0,1,0]=4$ to produce a Hamiltonian path $(r, \xi)$  of $\Q_n^{top}- \{r_1,g_1\}$ that connects $r$ to $g$ and
let $\xi = \mu \eta $, with $r \mu = g_2.$ Use $[1,1,0,1]=2$ to produce a Hamiltonian path $(r \mu' v, \theta)$ of
$\Q_n^{bot}-\{r_2\}$ that connects $r \mu' v $ to $r \mu \varphi(\eta)v.$ The desired Hamiltonian path of $\Q_n -
\{r_1,r_2,g_1,g_2\}$ for this case is $(r, \mu' v \theta v \eta^*).$

\emph{Subcase A2.} $g$ is on the top plate and $r$ is on the bottom plate.

Let $r_3$ be a red vertex on the top plate at a distance at least three away from $g_2.$ Use $[2,0,1,0]=4$ to produce a
Hamiltonian path $(g, \xi )$ of $ \Q_n^{top} - \{r_1,g_1\}$ that connects $g$ to $r_3.$ Let $ \xi = \mu \eta $, with $g
\mu = g_2$ and $\eta $ of length at least three. Use $[1,1,1,1]=4$ to produce a $2-$path covering of
$\Q_n^{bot}-\{r_2\}$ with paths $(g \mu' v, \theta), (r, \nu)$ connecting $g \mu' v$ to $g \mu \varphi(\eta) v$ and $r$
to $r_3 v$, respectively. The desired Hamiltonian path of $\Q_n - \{r_1,r_2,g_1,g_2\}$ for this case is $(g, \mu' v
\theta v \eta^* v \nu^R ).$

\emph{Subcase A3.} $r$ is on the top plate and $g$ is on the bottom plate.

Let $r_3$ be a red vertex on the top plate which is not adjacent to $g.$ Use $[3,1,0,1]=4$ to produce a Hamiltonian path
$(r, \xi )$ of $ \Q_n^{top} - \{r_1,g_1,g_2\}$ that connects $r$ to $r_3.$ Use $[1,1,0,1]=2$ to produce a Hamiltonian
path $(r_3 v, \mu)$ of $\Q_n^{bot}-\{r_2\}$ connecting $r_3 v$ to $g.$ Then the desired Hamiltonian path of $\Q_n -
\{r_1,r_2,g_1,g_2\}$ for this case is $(r, \xi v \mu).$

\emph{Subcase A4.} $r$ and $g$ are both on the bottom plate.

Let $r_3$ and $r_4$ be two red vertices on the top plate such that $r_3 v$ and $r_4 v$ are different from $g.$ Use
$[3,1,0,1]=4$ to produce a Hamiltonian path $(r_3,\xi)$ of $\Q_n^{top} - \{r_1,g_1,g_2\}$ that connects $r_3$ to $r_4.$
Use $[1,1,1,1]=4$ to produce a $2-$path covering of $\Q_n^{bot} - \{r_2\}$ with paths $(r, \eta)$ and $(r_4 v, \mu)$
connecting $r$ to $r_3 v$ and $r_4 v$ to $g$, respectively. The desired Hamiltonian path of $\Q_n - \{r_1,r_2,g_1,g_2\}$
for this case is $(r, \eta v \xi v \mu).$

\emph{Case B.} Each plate contains one deleted green vertex. We can assume that $g_1$ is on the top plate and $g_2$ is
on the bottom plate.

\emph{Subcase B1.} The two terminals are on the top plate.

Use $[2,0,1,0]=4$ to produce a Hamiltonian path $(r, \xi)$ of $\Q_n^{top} - \{r_1,g_1\}$ that connects $r$ to $g.$ Since
$n-1 \geq 4$ there exist words $\mu$ and $\eta$ and a letter $x$ such that $\xi = \mu x \eta $ with neither $r \mu v $
nor $r \mu x v $ being a deleted vertex. Use again $[2,0,1,0]=4$ to produce a Hamiltonian path $(r \mu v, \zeta)$ of
$\Q_n^{bot} -\{r_2, g_2\}$ that connects $r \mu v$ to $r \mu x v.$ The desired Hamiltonian path of $\Q_n -
\{r_1,r_2,g_1,g_2\}$ for this case is $(r, \mu v \zeta v \eta ).$

\emph{Subcase B2.} $g$ is on the top plate and $r$ is on the bottom plate.

Let $r_3$ be any red vertex on the top plate such that $r_3 v \ne g_2.$ Use $[2,0,1,0]=4$ to produce a Hamiltonian path
$(g, \xi )$ of $\Q_n^{top} - \{r_1, g_1\}$ connecting $g$ to $r_3$. Use again $[2,0,1,0]=4$ to produce a Hamiltonian
path $(r_3 v, \eta)$ of $\Q_n^{bot} - \{r_2, g_2\}$ connecting $r_3 v$ to $r$. The desired Hamiltonian path in $\Q_n -
\{r_1,r_2,g_1,g_2\}$ for this case is $(g, \xi v \eta).$
\end{proof}

\begin{lemma} $([0,0,3,0]= 5)$ \label{threecoveringlemma}
Let $n \ge 5$, $r_1,r_2,r_3$ be three distinct red vertices and
$g_1,g_2,g_3$ be three distinct green vertices in $\Q_n.$ Then
there exists a $3-$path covering of $\Q_n$ with paths $\gamma_i$
connecting $r_i $ to $g_i$ for $i=1,2,3.$ The claim is not true if
$n = 3$ or $n = 4$.
\end{lemma}

\begin{proof}
Let $r_1 = (0,0,0,0)$, $r_2 = (0,1,0,1)$, $r_3 = (0,1,1,0)$, $g_1 = (0,1,1,1)$, $g_2 = (0,0,1,0)$, and $g_3 = (0,0,0,1)$
be vertices in $\Q_4$. Then it is not difficult to verify that a $3-$path covering of $\Q_4$ with paths $\gamma_i$
connecting $r_i $ to $g_i$ for $i=1,2,3$ does not exist (see also \cite[Fig.1]{DGK}).

Let $n \geq 5$. Choose two plates to split the deleted red vertices such that $r_1$ and $r_2$ are on $\Q_n^{top}$ and
$r_3$ is on $\Q_n^{bot}$. There are five substantially different cases depending on the distribution of the green
terminals on the plates.

\emph{Case 1.} The three green terminals are on the top plate.

Use $[0,0,2,0]=2$ to produce a path covering $(r_1,\xi)$, $(r_2, \eta)$ of $\Q_n^{top}$ that connects $r_1$ to $g_1$ and
$r_2$ to $g_2.$ Without loss of generality we may assume that $g_3$ lies on the path between $r_2$ and $g_2.$ Let $\eta
= \mu \theta$, where $r_2 \mu = g_3$.

If $g_3 v \neq r_3$ then use $[0,0,0,2]=4$ to produce a path covering $(r_2 \mu' v, \nu)$, $(g_3 v, \zeta)$ of
$\Q_n^{bot}$ that connects $r_2 \mu' v $ to $g_2 (\theta^R)'v $ and $g_3 v$ to $r_3.$ The desired $3-$path covering for
this case is $(r_1, \xi)$, $(r_2, \mu' v \nu v \theta^*)$, $(r_3, \zeta^R v).$

If $g_3 v = r_3$ then use $[1,1,0,1]=2$ to produce a Hamiltonian path $(r_2 \mu' v, \nu)$ of $\Q_n^{bot}- \{r_3\}$ that
connects $r_2 \mu' v $ to $g_2 (\theta^R)'v $. The desired $3-$path covering for this case is $(r_1, \xi)$, $(r_2, \mu'
v \nu v \theta^*)$, $(r_3, r_3 v).$

\emph{Case 2.} Two green terminals are on the top plate and one is on the bottom plate.

If the green terminal on the bottom plate is $g_3$ then use $[0,0,2,0]=2$ to produce a $2-$path covering of $\Q_n^{top}$
connecting $r_1$ to $g_1$ and $r_2$ to $g_2$ and use $[0,0,1,0]=1$ to produce a Hamiltonian path of $\Q_n^{bot}$
connecting $r_3$ to $g_3.$

Now, assume that $g_1$ and $g_3$ are on the top plate and $g_2$ is on the bottom plate.

If $r_2 v \neq g_2$ and $g_3 v \neq r_3$ then use $[2,0,1,0]=4$ to find a Hamiltonian path $(r_1,\xi)$ of $\Q_n^{top}-
\{r_2,g_3\}$ connecting $r_1$ to $g_1$ and use $[0,0,0,2]=4$ to produce a $2-$path covering $(r_2 v, \eta)$, $(r_3,
\zeta)$ of $\Q_n^{bot}$ that connects $r_2 v$ to $g_2$ and $r_3$ to $g_3 v$. The desired $3-$path covering for this case
is $(r_1,\xi)$, $(r_2,  v \eta)$, $(r_3, \zeta v).$

Let $r_2 v \neq g_2$ and $g_3 v = r_3$ (the case $r_2 v = g_2$ and $g_3 v \neq r_3$ is symmetrical). Use $[2,0,1,0]=4$
to find a Hamiltonian path $(r_1,\xi)$ of $\Q_n^{top}- \{r_2,g_3\}$ connecting $r_1$ to $g_1$ and use $[1,1,0,1]=2$ to
produce a Hamiltonian path $(r_2 v, \eta)$ of $\Q_n^{bot}- \{r_3\}$ that connects $r_2 v$ to $g_2$. The desired $3-$path
covering for this case is $(r_1,\xi)$, $(r_2, v \eta)$, $(r_3, v).$

Finally, let $r_2 v = g_2$ and $g_3 v = r_3$. Use $[2,0,1,0]=4$ to find a Hamiltonian path $(r_1,\xi)$ of $\Q_n^{top}-
\{r_2,g_3\}$ connecting $r_1$ to $g_1$. Clearly, the length of the path $(r_1,\xi)$ is more than $1$. Use $[2,0,1,0]=4$
to find a Hamiltonian path $(r_1 v,\eta)$ of $\Q_n^{bot}- \{r_3,g_2\}$ connecting $r_1 v$ to $r_1 \varphi(\xi) v$. The
desired $3-$path covering for this case is $(r_1,v \eta v \xi^*)$, $(r_2, v)$, $(r_3,v).$

\emph{Case 3.} $g_3$ is on the top plate and the other two green terminals are on the bottom plate.

If $r_3 v = g_3$ then use $[1,1,0,1]=2$ to find a Hamiltonian path $(r_1, \xi )$ of $\Q_n^{top} - \{g_3\}$ that connects
$r_1$ to $r_2.$ Let $\xi = \mu x \eta $, with neither $r_1\mu v$ nor $r_1 \mu x v$ being a prescribed end. On the bottom
plate use $[1,1,1,1]=4$ to produce a $2-$path covering $(r_1 \mu v, \theta)$, $(r_1 \mu xv, \zeta)$ of $\Q_n^{bot}-
\{r_3\}$ connecting $r_1\mu v$ to $g_1$ and $r_1 \mu x v $ to $g_2$, respectively. The desired $3-$path covering for
this case is $(r_1, \mu v \theta)$, $(r_2, \eta^R v \zeta)$, $(r_3, v ).$

If $r_3 v \neq g_3$ use Corollary \ref{twochargedpathsonelong} to produce a $2-$path covering $(g_3, \xi)$ and $(r_1,
\eta)$ of the top plate with the first path connecting $g_3$ to $r_3 v$ and the second path of length at least $8$
connecting $r_1$ to $r_2.$ Let $\eta = \mu x \theta$, with neither $r_1 \mu v $ nor $r_1 \mu x v $ being a prescribed
end. Use $[1,1,1,1]=4$ to produce a $2-$path covering $(r_1 \mu v, \nu)$, $(r_1 \mu x v, \zeta)$ of $\Q_n^{bot}-\{r_3\}$
connecting $r_1 \mu v $ to $g_1$ and $r_1\mu x v $ to $g_2$, respectively. The desired $3-$path covering of $\Q_n$ for
this case is $(r_1, \mu v \nu)$, $(r_2, \theta^R v \zeta)$, $(r_3, v \xi^R).$

\emph{Case 4.} Either $g_1$ or $g_2$ is on the top plate and the other two green terminals are on the bottom plate.

Without loss of generality we can assume that $g_1$ is the green terminal on the top plate. Let $g_4$ be any green
vertex on the top plate such that $g_4 v$ is not a terminal vertex. Use $[0,0,2,0]=2$ to find a $2-$path covering $(r_1,
\xi)$, $(r_2, \eta)$ of $\Q_n^{top}$ that connects $r_1$ to $g_1$ and $r_2 $ to $g_4$, and a $2-$path covering $(r_3,
\mu)$, $(r_2 \eta v, \nu)$ of $\Q_n^{bot}$ that connects $r_3$ to $g_3$ and $r_2\eta v$ to $g_2.$ The desired $3-$path
covering for this case is $(r_1, \xi)$, $(r_2, \eta v \nu)$, $(r_3, \nu).$

\emph{Case 5.}  All the green terminals are on the bottom plate.

Let $r_4=g_1 x$ be any vertex on the bottom plate adjacent to $g_1$ and different from $r_3$. Use $[2,0,1,0]=4$ to
produce a Hamiltonian path $(g_2, \xi)$ of $\Q_n^{bot}- \{g_1, r_4\}$ that connects  $g_2$ to $r_3.$ Let $\xi = \mu \eta
$, with $g_2 \mu = g_3.$ Use $[0,0,2,0]=2$ to produce a $2-$path covering $(r_4 v,\theta)$, $(g_2 \mu' v, \zeta)$ of
$\Q_n^{top}$ connecting $r_4 v$ to $r_1$ and $g_2 \mu' v$ to $r_2$, respectively. The desired $3-$path covering of
$\Q_n$ for this case is $(g_1,xv\theta)$, $(g_2, \mu' v \zeta)$, $(g_3,\eta).$
\end{proof}

\section{Some general results}

Let $\mathcal{G}$ be a graph and $v$ be a vertex in $\mathcal{G}.$ We denote by $\mathcal{N}(v)$ the set of vertices
adjacent to $v$ in $\mathcal{G}.$ If $A$ is a subset of the set of vertices  of $ \mathcal{G}$ then the set
$\mathcal{N}(A) = \bigcup_{v\in A} \mathcal{N}(v)$ is called the set of neighbors of $A.$

As usual, if $X$ is a set, $|X|$ denotes the cardinality of $X.$

\begin{proposition} \label{neighborsofacluster} Let $A \subset \mathcal{N}(r)$
for some vertex $r$  in $\Q_n.$ Then $|\mathcal{N}(A) | = 1 + n |A| - \frac{|A| (|A|+1)}{2}.$
\end{proposition}

\begin{proof}  Obviously $r \in \mathcal{N}(A).$
Any pair of elements $g_1,g_2 \in A$ has exactly two neighbors in common  one of which is the root $r,$ and the other is
different for different pairs. It follows that $$|\mathcal{N}(A) | = 1 + (n-1) + (n-2) + \cdots + (n-|A|).$$
\vskip-18pt
\end{proof}

The following lemma is a particular case of an isoperimetric inequality
for the hypercube. See \cite[Theorem 7.3]{bezrukov} for a more
general statement and a discussion of several proofs available in
the literature. Here we just state and prove what we need in the sequel.

\begin{lemma} \label{neighborscardinality}
Let $k$ and $n$ be positive integers such that $1 \le k \le n$ and let $A$ be a set of green vertices in $ \Q_n$ of
cardinality $k.$ Then $$|\mathcal{N}(A)| \ge 1 + (n-1) + \cdots +(n-k) = 1 + k n - \frac{k (k+1)}{2},$$ with equality if
and only if $A \subset \mathcal{N}(r)$ for some red vertex $r.$
\end{lemma}

\begin{proof}
The statement is obvious for all pairs $k,n$ with $ 1 \le k \le 2 $ and $ k \le n. $ Let $N$ be a positive integer
greater than $2$ such that  the statement is true for all pairs $ k,n $ with $1 \le k \le n $ and $ n < N .$ We shall
prove that the statement is also true for all pairs $k, N$ with $ 1 \le k \le N.$

We split $\Q_N$ into two plates such that $ 1 \le l =|A \cap Q_N^{bot} | \le m = |A \cap \Q_N^{top}| \le N-1.$ Let
$A^{top} = A \cap \Q_N^{top}$ and $ A^{bot} = A \cap \Q_N^{bot}.$ Each element of $A^{top}$ has exactly one neighbor  in
$\Q_N^{bot}.$ Therefore, by Proposition \ref{neighborsofacluster} and  the induction hypothesis,
$$ |\mathcal{N}(A^{top} )| \ge 1 + [(N-1)-1 ] + \cdots + [(N-1)-m] +
m $$  $$  = 1 + (N-1) + \cdots + (N-m) ,$$ with equality throughout if and only if there exists $ r \in \Q_{N}^{top} $
such that $ A^{top} \subset \mathcal{N}(r).$

Similarly, let $s$ be the number of elements of $\mathcal{N}(A^{bot})$ that are in the top plate but not in $
\mathcal{N}(A^{top}).$ Then
$$ |\mathcal{N}(A^{bot} ) \setminus \mathcal{N}(A^{top})| \ge -m +  1 + [(N-1)-1 ] + \cdots + [(N-1)-l] + s $$
$$ \ge [(N-m) -1 ] + \cdots + [(N-m) - l ] ,$$
with equality throughout if and only if $ l =1$ and $s=0.$ It follows that $|\mathcal{N}(A)| \ge 1 + N-1 + \cdots + N- k
$ with equality if and only if there exists a vertex $r \in \Q_N$ such that $ A \subset \mathcal{N}(r).$
\end{proof}

\begin{lemma} \label{neutraltocharged}
Let $ M,C,N,O$ be nonnegative integers with $C, O,$ and $M$ of the same parity, $C \le M $, $O \ge C$, and $N \ge 1.$
Let also $k $ be a positive integer such that
\begin{equation}\label{criticalinequality}
kN +1 - { N+1\choose 2} > \frac{M+C}{2} + N + O.
\end{equation}
Then, $k\in \mathcal{A}_{M+1,C+1,N-1,O+1}$ implies $k \in \mathcal{A}_{M,C,N, O}.$
\end{lemma}

\begin{proof}
Let $k \in \mathcal{A}_{M+1,C+1,N-1,O+1}.$ This means that if $n \ge k$ then for every fault $\F$ of mass $M+1$ and
charge $C+1$ in $ \Q_n$ one can freely prescribe ends for a path covering of $\Q_n - \F$ with $N-1$ neutral paths and
$O+1$ charged paths. Consider an arbitrary fault $\F$ of mass $M$ and charge $C$ in $\Q_k,$ and a set $\mathcal{E}$ of
pairs of vertices that contains $N$ neutral pairs and $O$ charged pairs, and  is in balance with $\F.$

Without loss of generality we may assume that in $\F$ there are at least as many red vertices as there are green
vertices. It is easy to see that the number of the deleted green vertices is $\frac{M-C}{2},$ and that the number of
paths with green terminals at both ends is $\frac{O+C}{2}.$ Thus, the quantity $\frac{M+C}{2} + N + O$ is the total
number of green vertices that are either deleted vertices or terminal vertices.

The number of red terminals in neutral pairs is obviously $N.$ By Lemma \ref{neighborscardinality} the number of green
vertices that are adjacent to at least one red terminal in a neutral pair is at least $kN +1 - { N+1 \choose 2}.$
Therefore, inequality (\ref{criticalinequality}) guarantees the existence of a neutral pair $(r,g) \in \mathcal{E}$ and
a green vertex $g'=rx$ that is neither a deleted vertex nor a terminal vertex. The fault $\F' = \F \cup \{r\}$ has mass
$M+1$ and charge $C+1.$ The set of pairs of vertices $\mathcal{E'}$ obtained from $\mathcal{E}$ by replacing the pair
$(r,g)$ with the pair $(g',g) $ is in balance with $\F'$ and contains $N-1$ neutral pairs and $O+1$ charged pairs.
Therefore, there exists an $N+O-$path covering of $\Q_k-\F$ whose set of pairs of end vertices coincide with
$\mathcal{E'}.$ One of the paths in this covering is of the form $(g, \xi)$ with $g\xi = g'.$ If we replace this path
with the path $(g, \xi x )$ that connects $g$ to $r$ we obtain an $N+O-$path covering of $\Q_k - \F $ whose set of pairs
of end vertices coincides with $\mathcal{E}.$ So, we proved that for every fault $\F$ of mass $M$ and charge $C$ in $
\Q_k$ one can freely prescribe ends for a path covering of $\Q_k - \F$ with $N$ neutral paths and $O$ charged paths.
Finally, if $n \ge k $ then 1) $ nN +1 - { N+1 \choose 2} > \frac{M+C}{2} + N + O, $ and 2) $ n \in
\mathcal{A}_{M+1,C+1,N-1,O+1}.$ Therefore, the argument that we applied to $k$ can be applied to $n$ as well. This shows
that if $n \ge k$ then for every fault $\F$ of mass $M$ and charge $C$ in $ \Q_n$ one can freely prescribe ends for a
path covering of $\Q_n - \F$ with $N$ neutral paths and $O$ charged paths. Consequently $k \in \mathcal{A}_{M,C,N,O}.$
\end{proof}

\begin{lemma} \label{redchargedtoneutral}
Let $M,C,N,O$ be nonnegative integers with $C, O,$ and $M$ of the same parity, $C \le M $, and $O > C.$ Let also $k$ be
a positive integer such that
\begin{equation}\label{criticalinequality2}
k(O-C) +1 - {O-C+1\choose 2} > \frac{M+C}{2} + N + O.
\end{equation}
Then,  $k \in \mathcal{A}_{M+1,C+1,N+1,O-1}$ implies $k \in \mathcal{A}_{M,C,N, O}.$
\end{lemma}

\begin{proof}
The proof is similar to the proof of Lemma \ref{neutraltocharged}. The only difference is that in
(\ref{criticalinequality2}), instead of $N$, we use the number $O-C$ that represents the number of red terminals in the
charged paths.
\end{proof}

\begin{lemma} \label{greenchargedtoneutral}
Let $M,C,N,O$ be nonnegative integers with $C, O,$ and $M$ of the same parity, $C \le M $, $O \ge C$, and $C \ge 1.$ Let
also $k$ be a positive integer such that
\begin{equation}\label{criticalinequality3}
k(O+C) +1 - {O+C+1\choose 2} > \frac{M-C}{2} + N + O.
\end{equation}
Then, $k\in \mathcal{A}_{M+1,C-1,N+1,O-1}$ implies $k \in \mathcal{A}_{M,C,N,O}.$
\end{lemma}

\begin{proof}
The proof is similar to the proof of Lemma \ref{neutraltocharged}. The difference is that in the left-hand side of
(\ref{criticalinequality3}), instead of $N$, we use the number $O+C$ that represents the number of green terminals in
the charged paths and the right-hand side part $\frac{M-C}{2} + N + O$ represents the number of red vertices that are
either in $\F$ or are terminals.
\end{proof}

\begin{lemma}\label{l4202}
$[4,2,0,2]= [3,1,1,1] = 5 $ and $ [2,0,2,0]= 4.$
\end{lemma}

\begin{proof} It follows from Lemma \ref{neutraltocharged} that if $5 \in
\mathcal{A}_{4,2,0,2}$ then $5$ is in $\mathcal{A}_{3,1,1,1}$ and in $\mathcal{A}_{2,0,2,0}.$  Lemma
\ref{neutralpathsinq4}, proved in Appendix A, states that we can freely prescribe  two neutral pairs of terminals for a
$2-$path covering of $ \Q_4 - \F$ for any neutral fault of mass $2.$ Therefore, to prove the current lemma, it is
sufficient to show that $5 \in \mathcal{A}_{4,2,0,2}$, $4 \not \in \mathcal{A}_{3,1,1,1}$ (and therefore, according to
Lemma \ref{neutraltocharged}, $4 \not \in \mathcal{A}_{4,2,0,2}$), and that $3 \notin \mathcal{A}_{2,0,2,0}$.

Here is a counterexample showing that $3 \notin \mathcal{A}_{2,0,2,0}$. Let $n=3,$ $r_1=(1,0,0),$ $g_1= (0,1,1)$, $r_2
=(0,1,0),$ $g_2=(1,0,1)$, and $\F = \{(0,0,0), (1,1,1)\}$. Then, a $2-$path covering of $\Q_3 - \F$ that connects $r_1$
to $g_1$ and $r_2$ to $g_2$ does not exist.

The following counterexample shows that $4 \not \in \mathcal{A}_{3,1,1,1}$ (see also the discussion after Conjecture
\ref{chargeonetwopathsconjecture}).

Let $n=4,$ $ \F = \{(0,0,0,0), (0,1,0,1), (0,1,1,1)\}$, $r_1=(1,1,0,0),$ $ g_1= (1,0,0,0), g_2 =(0,0,1,0),$ and $g_3
=(1,1,1,0).$ Then, a $2-$path covering of $\Q_4 - \F$ that connects $r_1$ to $g_1$ and $g_2$ to $g_3$ does not exist.

We now prove that $5 \in \mathcal{A}_{4,2,0,2}.$ Let $n \ge 5.$ We can assume that  $\F =\{r_1,r_2,r_3,g\}$ with
$r_1,r_2,r_3$ being red and $g$ being a green vertex. Let  $\mathcal{E} = \{(g_1,g_2),(g_3,g_4)\}$ be the set of pairs of green end
vertices. We are looking for $2-$path coverings of $\Q_n-\F$ with paths that connect $g_1$ to $g_2$ and $g_3$ to $g_4.$
We split $\Q_n$ into two plates with two red vertices in the top plate, say $r_1$ and $r_2$, and $r_3$ in the bottom
plate. Then we consider a group of cases when the green deleted vertex $g$ is on the top plate and another group of
cases when the green deleted vertex is on the bottom plate. The cases within each group depend on the distribution of
the green terminals on the plates.

\emph{Case A.} The green deleted vertex is on the top plate.

\emph{Subcase A1.} All the green terminals are on the top plate.

Let $(g_1,\xi)$ be a Hamiltonian path on $\Q_n^{top} - \{r_1,r_2,g\}$ that connects $g_1$ to $g_2.$ Such path exists
since $[3,1,0,1]=4.$ Let $\xi = \eta \theta \mu $ with $g_1 \eta = g_3$ and $g_1 \eta \theta = g_4,$ where $g_3, g_4$
are renumbered, if necessary. Let $(g_1 \xi' v, \zeta)$ be a Hamiltonian path on $\Q_n^{bot} -\{r_3\}$ that connects
$g_1 \xi' v $ to $g_2(\mu^R)'v.$ Such path exists since $[1,1,0,1]=2.$ The desired $2-$path covering of $\Q_n - \F$ for
this case is $(g_1,\xi' v \zeta v \mu^*), (g_3, \theta).$

\emph{Subcase A2.} $g_1,g_2,g_3$ are on the top plate and $g_4$ is on the bottom plate.

Let $(g_1,\xi)$ be a Hamiltonian path on $\Q_n^{top} - \{r_1,r_2,g\}$ that connects $g_1$ to $g_3.$ Such path exists
since $[3,1,0,1]=4.$ Let $\xi = \eta \theta$ with $g_1 \eta = g_2.$ Let $(g_3 (\theta^R)' v, \zeta)$ be a Hamiltonian
path on $\Q_n^{bot} -\{r_3\}$ that connects $g_3 (\theta^R)' v$ to $g_4.$ Such path exists since $[1,1,0,1]=2.$ The
desired $2-$path covering of $\Q_n - \F$ for this case is $(g_1,\eta), (g_3, (\theta^R)' v \zeta).$

\emph{Subcase A3.} $g_1,g_2$ are on the top plate and $g_3,g_4$ are on the bottom plate.

We simply connect $g_1$ to $g_2$ by a Hamiltonian path of $\Q_n^{top} - \{r_1,r_2,g\}$ and $g_3$ to $g_4$ by a
Hamiltonian path of $\Q_n^{bot} - \{r_3\}.$ That produces the desired $2-$path covering of $\Q_n - \F$ for this case.

\emph{Subcase A4.} $g_1,g_3$ are on the top plate and $g_2,g_4$ are on the bottom plate.

Let $(g_1,\xi)$ be a Hamiltonian path on $\Q_n^{top} - \{r_1,r_2,g\}$ that connects $g_1$ to $g_3.$ Such path exists
since $[3,1,0,1]=4.$ We can find words $\eta, \theta,$ and a letter $x$ such that $\xi = \eta x \theta,$ and neither
$g_1\eta v$ nor $g_1 \eta x v$ is a deleted vertex or a terminal. Let $(g_1\eta v, \mu), (g_1 \eta x v, \nu)$ be a
$2-$path covering of $\Q_n^{bot} -\{r_3\}$ that connects $g_1 \eta v$ to $g_2$ and $g_1 \eta x v$  to $ g_4.$ Such path
covering exists since $[1,1,1,1]=4.$ The desired $2-$path covering of $\Q_n - \F$ for this case is $(g_1,\eta v \mu),
(g_3, \theta^R v \nu).$

\emph{Subcase A5.} $g_1$ is on the top plate and $g_2,g_3,g_4$ are on the bottom plate.

Let $r \neq r_3 $ be a red vertex on the bottom plate such that $rv \neq g_1, g.$ Let $(g_2, \eta), (g_3, \theta)$ be a
$2-$path covering of $\Q_n^{bot}-\{r_3\}$ that connects $g_2$ to $r$ and $g_3$ to $g_4.$ Such path covering exists since
$[1,1,1,1]=4.$ Let $(g_1,\mu)$ be a Hamiltonian path of $\Q_n^{top}-\{r_1,r_2,g\}$ that connects $g_1$ to $rv.$ The
desired $2-$path covering of $\Q_n - \F$ for this case is $(g_1, \mu v \eta^R), (g_3, \theta).$

\emph{Subcase A6.} All the green terminals are on the bottom  plate.

First we assume that either $g_3$ or $g_4$ (or, equivalently, $g_1$ or $g_2$) is not adjacent to $g v.$ Without loss of
generality we can assume that $g_3$ is at distance at least three from $g v$ and let $x$ be a letter such that $g_2 x v
\neq g.$ Let $(g_1, \xi )$ be a Hamiltonian path of $\Q_n^{bot}- \{r_3, g_2, g_2x\}$ that connects $g_1$ to $g_4.$ Such
path exists since $[3,1,0,1]=4.$ Then $\xi =\eta \theta$ with $g_1 \eta = g_3.$ Observe that our assumption on $g_3$
guarantees that $g_1 \eta'v \neq g.$ Let $(g_1 \eta' v ,\zeta)$ be a Hamiltonian path on $ \Q_n^{top} - \{r_1,r_2,g\}$
that connects $g_1\eta' v$ to $g_2 x v.$ Such path exists since $[3,1,0,1]=4.$ The desired $2-$path covering of $\Q_n-
\F$ for this case is $(g_1, \eta' v \zeta vx), (g_3, \theta).$

Now let us assume that $g v = r_3$ and all the vertices $g_1$, $g_2$, $g_3$ and $g_4$ are adjacent to $g v$. Then we can
use the same construction as in the previous case to find the desired $2-$path covering. In this case the requirement
one of the green terminals to be at distance three from $g v$ is not necessary since $g v = r_3$.

Finally, let us assume that $g v\ne r_3$ and $g_3$ and $g_4$ are adjacent to $g v.$ This means that there exist letters
$x,y$ such that $g_3 x = g_4 y = g v.$ Let $(g_1, \xi) $ be a Hamiltonian cycle in $ \Q_n^{bot} - \{r_3, g_3, g_4, g
v\}.$ Such cycle exists since $[4]=4.$ Then $\xi = \eta \theta$ with $g_1 \eta = g_2.$ Let $(g_1\eta' v, \zeta)$ be a
Hamiltonian path of $\Q_n^{top}-\{r_1,r_2,g\}$ that connects $g_1 \eta' v$ to $g_1 \xi' v.$ Such path exists since
$[3,1,0,1]=4.$ The desired $2-$path covering of $\Q_n- \F$ for this case is $(g_1,\eta' v \zeta v (\theta')^R), (g_3,
xy).$

\emph{Case B.} The green deleted vertex is on the bottom plate.

\emph{Subcase B1.} All the green terminals are on the top plate.

Let $(g_1,\xi),(g_3,\eta)$ be a $2-$path covering of $\Q_n^{top}-\{r_1,r_2\}$ that connects $g_1$ to $g_2$ and $g_3$ to
$g_4.$ Such path covering exists since $[2,2,0,2]=4.$ Without loss of generality we can assume that the word $\xi$ is
not shorter than the word $\eta.$ Therefore, there exist words $\mu,\nu$ and a letter $x$ such that $\xi = \mu x \nu$
with neither $g_1\mu v$ nor $g_1 \mu x v$ being a deleted vertex. Let $(g_1 \mu v,\zeta)$ be a Hamiltonian path of
$\Q_n^{bot} -\{r_3, g\}$ that connects $g_1 \mu v$ to $g_1 \mu x v.$ Such path exists since $[2,0,1,0]=4.$ The desired
$2-$path covering of $\Q_n- \F$ for this case is $(g_1, \mu v \zeta v \nu), (g_3,\eta).$

\emph{Subcase B2.} $g_1,g_2,g_3$ are on the top plate and $g_4$ is on the bottom plate.

Let $g_5$ be a green vertex on the top plate such that $g_5 v$ is not a deleted vertex. Let $(g_1, \xi),(g_3,\eta)$ be a
$2-$path covering of $\Q_n^{top} - \{r_1, r_2\}$ that connects $g_1$ to $g_2$ and $g_3$ to $g_5.$ Such path covering
exists since $[2,2,0,2]=4.$ Let $(g_5 v, \zeta)$ be a Hamiltonian path of $\Q_n^{bot}- \{r_3,g\}$ that connects $g_5 v$
to $g_4.$ Such path exists since $[2,0,1,0]=4.$ The desired $2-$path covering of $\Q_n- \F$ for this case is $(g_1,
\xi), (g_3, \eta v \zeta).$

\emph{Subcase B3.} $g_1,g_2$ are on the top plate and $g_3,g_4$ are on the bottom plate.

Since $n \ge 4$ we can find words $\eta, \theta$ of length greater than three such that $(g_3, \eta \theta)$ is a
Hamiltonian cycle of $\Q_n^{bot}-\{r_3,g\}$ with $g_3 \eta = g_4$ (Lemma \ref{new}). For at least one of the four pairs
of green vertices $(g_3 \varphi(\eta)v,g_3 \eta'v),$ $(g_3 \varphi(\eta)v , g_4 \varphi(\theta)v),$ $(g_3 \eta'v, g_4
\theta'v), (g_4 \varphi(\theta)v, g_4 \theta'v)$ the two elements in the pair are not terminals on the top plate.

Assume that neither $g_3 \varphi(\eta)v$ nor $g_3\eta'v$ is a terminal vertex. Let $(g_1,\mu),$ $(g_2,\nu)$ be a
$2-$path covering of $\Q_n^{top}-\{r_1,r_2\}$ that connects $g_1$ to $g_3 \varphi(\eta)v$ and $g_2$ to $g_3 \eta' v.$
Such path covering exists since $[2,2,0,2]=4.$ The desired $2-$path covering of $\Q_n- \F$ for this case is $(g_1,\mu v
\eta'^* v\nu^R),$ $(g_3, \theta^R).$

The case when neither $g_4 \varphi(\theta)v$ nor $g_4\theta'v$ is a terminal vertex is equivalent to the previous case.

Assume now that neither $g_3 \varphi(\eta)v$ nor $g_4\varphi(\theta)v$ is a terminal vertex.  Let $(g_1, \mu),(g_4
\varphi(\theta) v,\nu)$ be a $2-$path covering of $\Q_n^{top}-\{r_1,r_2\}$ that connects $g_1$ to $g_2$ and $g_4
\varphi(\theta) v$ to $g_3 \varphi(\eta)v.$ Such path covering exists since $[2,2,0,2]=4.$ The desired $2-$path covering
of $\Q_n- \F$ for this case is $(g_1,\mu),(g_3,(\theta^R)'v\nu v \eta^*).$

The case when neither $g_3 \eta'v $ nor $g_4\theta'v$ is a terminal vertex is equivalent to the previous case.

\emph{Subcase B4.} $g_1,g_3$ are on the top plate and $g_2,g_4$ are on the bottom plate.

Let $x$ be a letter such that $g_2 x \neq r_3$ and $g_2 x v \neq g_1, g_3.$ Such letter exists since the dimension of
the plates is greater than or equal to $4.$ Let $(g_4, \xi)$ be a Hamiltonian cycle of $\Q_n^{bot}-\{r_3,g,g_2, g_2
x\}.$ Such cycle exists since $[4]=4.$ We can also assume that $g_4\xi'v \neq g_1$ by replacing $\xi$ with $\xi^R$, if
necessary.

Assume that $g_4 \xi' v = g_3.$ Let $(g_1, \mu)$ be a Hamiltonian path of $\Q_n^{top}-\{r_1,r_2,g_3\}$ that connects
$g_1$ to $g_2 x v.$ Such path exists since $[3,1,0,1]= 4.$ The desired $2-$path covering of $\Q_n- \F$ for this case is
$(g_1, \mu v x),$ $(g_4, \xi'v).$

Finally, if $g_4 \xi' v\neq g_3$ we proceed as follows. Let $(g_1,\mu), (g_3, \nu)$ be a $2-$path covering of
$\Q_n^{top}-\{r_1,r_2\}$ that connects $g_1$ to $ g_2 xv $ and $g_3$ to $g_4 \xi' v.$ Such path covering exists since
$[2,2,0,2]=4.$ The desired $2-$path covering of $\Q_n- \F$ for this case is $(g_1, \mu v x), (g_3, \nu v (\xi')^R).$

\emph{Subcase B5.} $g_1$ is on the top plate and $g_2,g_3,g_4$ are on the bottom plate.

Let $ x$ be a letter different from $v$ such that $ g_2x \ne r_3$ and $ g_2 x v \ne g_1. $ Let  $ ( g_3, \xi) $ be a Hamiltonian cycle of $ \Q_n^{bot} - \{r_3, g, g_2,g_2x\}$ ($[4]=4$).  $ \xi = \eta \zeta $ with   $ g_3 \eta = g_4.$  We can also assume, by renumbering the vertices and/or reversing the cycle if necessary, that $\eta$ has more than two letters and that $g_3 \eta' \ne g_1v. $

If $g_3 \varphi(\eta) $ is also different from $ g_1v $ then let $ ( g_1, \mu), (g_3 \eta' v , \nu) $ be a $2-$path covering of $ \Q_n^{top} -\{r_1,r_2\}$ that connects $g_1$ to $ g_3 \varphi(\eta) v,$ and $ g_3 \eta' v$ to $ g_2 x v $ ($[2,2,0,2]=4$). The desired $2-$path covering of $\Q_n- \F$  is
$(g_1, \mu v (\eta^*)' v \nu v x  )$, $(g_3, \zeta^R).$

If $g_3 \varphi(\eta) = g_1 v$ then let $ (g_3 \eta' v, \mu) $ be a Hamiltonian 
path of $\Q_n^{top} - \{r_1, r_2, g_1\}$ that connects $g_3 \eta' v$ to $g_2xv$ 
($[3,1,0,1]=4$). The desired $2-$path 
covering of $\Q_n- \F$ is $(g_1,v(\eta^*)'v\mu vx),(g_3,\zeta^R).$

\emph{Subcase B6.} All the green terminals are on the bottom  plate.

Let $(g_1, \xi)$ be a Hamiltonian cycle of $ \Q_n^{bot}-\{r_3,g\}.$ Such cycle exists for $[2]=3.$ Since the dimension
of the plates are greater than or equal to $4$ we can also assume that the distance from $g_1$ to $g_2$ along the cycle
is at least $4$ (Lemma \ref{new}). There are two essentially different distributions of the four green terminals along
the cycle. In the first case $\xi = \eta \theta \zeta \kappa$ with $g_1 \eta = g_2, g_2 \theta = g_3, g_3 \zeta = g_4,$
where $g_3, g_4$ are to be renumbered, if necessary. In the second case $\xi = \eta \theta \zeta \kappa$ with $g_1 \eta
= g_3, g_3 \theta = g_2, g_2 \zeta = g_4,$ where $g_3, g_4$ are to be renumbered, if necessary.

In the first case we proceed as follows. Let $(g_1 \varphi(\eta)v, \mu), (g_1 \eta'v, \nu)$ be a $2-$path covering of
$\Q_n^{top}-\{r_1,r_2\}$ that connects $g_1 \varphi(\eta)v$ to $g_1 (\kappa^R)'v$ and $g_1\eta'v$ to $g_2\theta'v.$ Then the
desired $2-$path covering of $\Q_n- \F$ for this case is $(g_1, (\kappa^R)' v \mu^R  v \eta'^* v \nu v (\theta')^R),$ $(g_3,\zeta).$

In the second case we proceed as follows. Let $(g_1 \eta' v, \mu),(g_3 \theta' v , \nu)$ be a $2-$path covering of
$\Q_n^{top}-\{r_1,r_2\}$ that connects $g_1\eta'v$ to $g_2 \zeta' v$ and $g_3 \theta' v$ to $g_4 \kappa' v.$  Such path
covering exists since $[2,2,0,2]=4.$ The desired $2-$path covering of $\Q_n- \F$ for this  case is $(g_1, \eta' v \mu v
(\zeta')^R),$ $(g_3, \theta' v \nu v (\kappa')^R).$
\end{proof}

\begin{lemma}\label{l2002}
$[2,0,0,2] = 5.$
\end{lemma}

\begin{proof}
It follows from Lemma \ref{redchargedtoneutral} that $[2,0,0,2]\leq [3,1,1,1]$ and 
since $[3,1,1,1]=5$ (Lemma \ref{l4202}) we have $[2,0,0,2]\leq 5.$ The following
counterexample shows that $[2,0,0,2]\geq 5$.

Let $r=(0,1,1,0)$, $r_1 = (0,0,1,1)$, $r_2 = (0,1,0,1)$, $g=(1,1,0,1)$, $g_1 = (1,0,1,1)$, $g_2 = (1,1,1,0)$ be vertices
in $\Q_4$. Then it is not difficult to verify that a $2-$path covering of $\Q_4-\{r,g\}$ with path $\gamma_1$ connecting
$r_1 $ to $r_2$ and path $\gamma_2$ connecting $g_1 $ to $g_2$ does not exist.
\end{proof}

\begin{lemma}$([5,1,0,1]=5)$ Let $n \geq 5$ and
$\F = \{r_1,r_2,r_3,g_1,g_2\}$ be a fault with three distinct red and two distinct green vertices. If $g_3,g_4\in \Q_n -
\F$ are two distinct green vertices then there exists a Hamiltonian path of $\Q_n - \F$ that connects $g_3$ to $g_4$.
The claim is not true if $n = 3$ or $n = 4$.
\end{lemma}

\begin{proof} It follows from Lemma \ref{neutraltocharged} that if $k \ge 4$
and $k \in \mathcal{A}_{5,1,0,1}$ then $k$ is in $\mathcal{A}_{4,0,1,0}$ and since $[4,0,1,0] = 5$ we have $[5,1,0,1]
\ge 5.$ We shall prove that $[5,1,0,1]=5.$ Let $n\ge 5.$ Split $\Q_n$ into two plates in a way that two red vertices,
say $r_1$ and $r_2$, are on the top plate and $r_3$ is on the bottom plate. We shall consider all essentially different
cases depending on the distribution of the two green deleted vertices and the two green terminals.

\emph{Case A.} The two green deleted vertices are on the top plate.

\emph{Subcase A1.} $g_3$ and $g_4$ are on the top plate.

Use $[4]=4$ to find a Hamiltonian cycle $(g_3, \xi)$ of $\Q_n^{top}- \{r_1,r_2,g_1,g_2\}.$ Let $\xi = \eta \theta $,
with $ g_3 \eta = g_4.$ Use $[1,1,0,1]= 2$ to find a Hamiltonian path $(g_3 \eta' v, \zeta)$ of $\Q_n^{bot}- \{r_3\}$
that connects $g_3 \eta' v$ to $g_3 \xi' v.$ The desired Hamiltonian path of $\Q_n - \F$ for this case is $(g_3, \eta' v
\zeta v (\theta')^R).$

\emph{Subcase A2.} $g_3$ is on the top plate and $g_4$ is on the bottom plate.

Use $[4]=4$ to find a Hamiltonian cycle $(g_3, \xi)$ of $\Q_n^{top}- \{r_1,r_2,g_1,g_2\}.$ Either $g_3\varphi(\xi)$ or
$g_3 \xi'$ is not adjacent to $g_4$. Assume, without loss of generality, that $g_3 \xi'$ is not adjacent to $g_4$. Use
$[1,1,0,1]= 2$ to find a Hamiltonian path $(g_3 \xi' v, \eta )$ of $ \Q_n^{bot}- \{r_3\}$ that connects $g_3 \xi' v$ to
$g_4.$ The desired Hamiltonian path of $\Q_n - \F$ for this case is $(g_3, \xi' v \eta ).$

\emph{Subcase A3.} $g_3$ and $g_4$ are on the bottom plate.

Use $[4]=4$ to find a Hamiltonian cycle $\gamma$ of $\Q_n^{top} - \{r_1,r_2,g_1,g_2\}.$ Let $a,b$ be two consecutive
vertices along this cycle such that neither $av$ nor $bv$ is a deleted vertex or a terminal and let $\gamma = (a, \xi),$
with $a\xi' = b$. Use $[1,1,1,1]=4$ to find a $2-$path covering $(av , \eta), (bv, \theta)$  of $\Q_n^{bot}- \{r_3\}$
that connects $av$ to $g_3$ and $bv$ to $g_4.$ The desired Hamiltonian path of $\Q_n - \F$ for this case is $(g_3,\eta^R
v  \xi' v \theta ).$

\emph{Case B.} $g_1$ is on the top plate and $g_2$ is on the bottom plate.

\emph{Subcase B1.} $g_3$ and $g_4$ are on the top plate.

Use $[3,1,0,1]=4$ to find a Hamiltonian path $(g_1, \xi)$ of $\Q_n^{top}- \{r_1, r_2, g_1\}$ that connects $g_3$ to
$g_4.$ Since $n \geq 5$ there exist words $\eta, \theta$ and a letter $x$ such that $\xi = \eta x \theta,$ and neither
$g_3 \eta v$ nor $g_3 \eta x v $ is a deleted vertex. Use $[2,0,1,0]=4$ to find a Hamiltonian path $(g_3 \eta v, \zeta)$
of $\Q_n^{bot} - \{r_3,g_2\}$ that connects $g_3 \eta v$ to $g_3 \eta x v.$ The desired Hamiltonian path of $\Q_n - \F$
for this case is $(g_3,\eta v \zeta v \theta).$

\emph{Subcase B2.} $g_3$ is on the top plate and $g_4$ is on the bottom plate.

Let $g_5$ be a green vertex on the top plate such that neither $g_5$ nor $g_5 v$ is a deleted vertex or a terminal. Use
$[3,1,0,1]=4$ to find a Hamiltonian path $(g_3, \xi)$ of $\Q_n^{top}- \{r_1, r_2, g_1\}$ that connects $g_3$ to $g_5.$
Use $[2,0,1,0]=4$ to find a Hamiltonian path $(g_5 v, \eta)$ of $\Q_n^{bot} - \{r_3,g_2\}$ that connects $g_5 v$ to
$g_4.$ The desired Hamiltonian path of $\Q_n - \F$ for this case is $ (g_3, \xi v \eta ).$

\emph{Subcase B3.} $g_3$ and $g_4$ are on the bottom plate.

Let $g_5$ and $g_6$ be any two green vertices on the top plate different from $g_1$ such that neither $g_5 v$ nor $g_6
v$ is a deleted vertex (clearly they cannot be terminal vertices). Use $[3,1,0,1]=4$ to find a Hamiltonian path $(g_5,
\xi)$ of $\Q_n^{top}- \{r_1, r_2, g_1\}$ that connects $g_5$ to $g_6.$ Use $[2,0,2,0]=4$ to find a $2-$path covering
$(g_3,\eta)$, $(g_4, \theta)$ of $\Q_n^{bot}- \{r_3,g_2\}$ that connects $g_3$ to $g_5 v$ and $g_4$ to $g_6 v.$ The
desired Hamiltonian path of $\Q_n - \F$ for this case is $(g_3,\eta v \xi v \theta^R ).$

\emph{Case C.} The two green deleted vertices are on the bottom plate.

\emph{Subcase C1.} $g_3$ and $g_4$ are on the top plate.

Let $g_5$ and $g_6$ be any two green vertices on the top plate different from $g_3$ and $g_4$ such that $g_5 v \neq r_3$
and $g_6 v \neq r_3$. Use $[2,0,2,0]=4$ to find a $2-$path covering $(g_3,\xi)$, $(g_4, \eta)$ of $\Q_n^{top}-
\{r_1,r_2\}$ that connects $g_3$ to $g_5$ and $g_4$ to $g_6.$ Use $[3,1,0,1]=4$ to find a Hamiltonian path $(g_5
v,\zeta)$ of $\Q_n^{bot}- \{r_3, g_1, g_2\}$ that connects $g_5 v$ to $g_6 v.$ The desired Hamiltonian path of $\Q_n -
\F$ for this case is $(g_3, \xi v \zeta v  \eta^R ).$

\emph{Subcase C2.} $g_3$ is on the top plate and $g_4$ is on the bottom plate.

Let $r_4$ be a red vertex on the bottom plate such that neither $r_4$ nor $r_4 v$ is a deleted vertex or a terminal. Use
$[4]=4$ to find a Hamiltonian cycle $(g_4, \xi )$ of $\Q_n^{bot} - \{r_3,r_4,g_1,g_2\}.$ By replacing $\xi$ with
$\xi^R$, if necessary, we can assume that $g_4 \xi' v \neq g_3.$ Since the bottom plate is of dimension at least $4,$
there exists a letter $y$ such that $g_5=r_4 y$ is neither a terminal nor a deleted vertex. Let $\xi = \eta \theta$ with
$g_4 \eta = g_5.$ Set $g_6 = g_4\eta' v $ or $g_6 = g_4 \eta \varphi(\theta) v,$ making sure that $g_6 \neq g_3.$ Use
$[2,2,0,2]=4$ to find a $2-$path covering $(g_3, \mu)$, $(g_6, \nu)$ of $\Q_n^{top}-\{r_1,r_2\}$ that connects $g_3$ to
$g_4 \xi' v$ and $g_6$ to $r_4 v.$ The desired Hamiltonian path of $\Q_n - \F$ for this case is $(g_4, \eta' v \nu v y
\theta' v \mu^R)$ if $g_6 = g_4 \eta' v $ or $(g_4, \eta y  v \nu^R v \theta'^* v \mu^R)$ if $g_6 = g_4 \eta
\varphi(\theta) v.$

\emph{Subcase C3.} $g_3$ and $g_4$ are on the bottom plate.

Let $r_4$ and $r_5$ be any two red vertices on the bottom plate that are not deleted vertices. Use $[3,1,0,1]=4$ to find
a Hamiltonian path $(r_4, \xi)$ of $\Q_n^{bot}- \{r_3, g_1, g_2\}$ that connects $r_4$ to $r_5$ and let $\xi = \eta
\theta \mu$, with $r_4\eta = g_3$ and $r_4\eta \theta = g_4$, where $g_3$ and $g_4$ should be renumbered, if necessary.
If the length of $\eta$ is at least three then use $[2,2,0,2]=4$ to find a $2-$path covering $(r_4\eta'v, \nu), (g_3
\theta' v, \zeta)$ of $\Q_n^{top}-\{r_1,r_2\}$ that connects $r_4\eta' v $ to $r_5 v$ and $g_3 \theta'v $ to $r_4 v.$
The desired Hamiltonian path of $\Q_n - \F$ for this case is $(g_3,\theta'v \zeta v \eta' v \nu v \mu^R ).$

The case when the length of $\mu$ is at least three is equivalent to the case when the length of $\eta$ is at least
three.

If $\eta$ and $\mu$ are both of length one then $\theta$ is of length greater than three. In this case use $[2,2,0,2]=4$
to produce a $2-$path covering $(r_4 v, \nu), (g_3\varphi(\theta)v,\zeta)$ of $\Q_n^{top}-\{r_1,r_2\}$ that connects
$r_4 v$ to $g_3 \theta' v$ and $g_3 \varphi(\theta)v$ to $r_5 v.$ The desired Hamiltonian path of $\Q_n - \F$ for this
case is $(g_3, \eta^R v \nu v (\theta'^*)^R v \zeta v \mu^R ).$
\end{proof}

\begin{lemma}\label{threechargedpaths}$([3,3,0,3]\le 6)$
Let $n\ge 6$ and $\F=\{r_1,r_2,r_3\}$ be a fault in $\Q_n$ with three distinct red vertices. If $g_1$, $g_2$, $g_3$,
$g_4$, $g_5$, $g_6$ are six distinct green vertices in $\Q_n- \F$ then there exists a $3-$path covering of $\Q_n- \F$ that
connects $g_1$ to $g_2,$ $g_3$ to $g_4,$ and $g_5$ to $g_6.$
\end{lemma}

\begin{proof} Split $\Q_n$ into two plates with two red vertices, say $r_1$
and $r_2$, on the top plate, and $r_3$ on the bottom plate. We consider several cases that depend on the distribution of
the green terminals on the plates.

\emph{Case 1.} All the green terminals are on the top plate.

Without loss of generality we can assume that $g_6v \neq r_3.$ Let $x$ be a letter such that $g_5 x$ is not a deleted
vertex. Let $(g_1, \xi), (g_3, \eta) $ be a $2-$path covering of $\Q_n^{top}-\{r_1, r_2, g_5, g_5x\} $ that connects
$g_1$ to $g_2$ and $g_3$ to $g_4.$ Such path covering exists since $[4,2,0,2]=5.$ Without loss of generality we can
assume that $g_6$ lies on the path from $g_3$ to $g_4.$ Let $\eta = \theta \zeta$ with $g_3 \theta = g_6$ and let
$(g_5xv, \mu), (g_3 \theta'v, \nu)$ be a $2-$path covering of $\Q_n^{bot} -\{r_3\}$ that connects $g_5xv$ to $g_6v$ and
$g_3\theta'v$ to $g_6 \varphi(\zeta)v.$ Such path covering exists since $[1,1,1,1]=4.$ The desired $3-$path covering of
$\Q_n- \F$ for this case is $(g_1,\xi), (g_3, \theta'v\nu v \zeta^*), (g_5, x v \mu v).$

\emph{Case 2.} $g_1,g_2,g_3,g_4,g_5$ are on the top plate and $g_6$ is on the bottom plate.

Let $x$ be a letter such that  $g_5x$ is not a deleted vertex and $g_5xv \neq g_6.$ Let $(g_1,\xi), (g_3,\eta)$ be a
$2-$path covering of $\Q_n^{top}-\{r_1,r_2,g_5,g_5x\}$ that connects $g_1$ to $g_2$ and $g_3$ to $g_4.$ Such path
covering exists since $[4,2,0,2]=5.$ Let $(g_5 x v, \mu)$ be a Hamiltonian path of $\Q_n^{bot}-\{r_3\}$ that connects
$g_5 xv$ to $g_6.$ Such path exists since $[1,1,0,1]=2.$ The desired $3-$path covering of $\Q_n- \F$ for this case is
$(g_1,\xi), (g_3, \eta), (g_5, x v \mu ).$

\emph{Case 3.}  $g_1,g_2,g_3,g_4,$ are on the top plate and $g_5, g_6$ are on the bottom plate.

Here we simply connect $g_1$ to $g_2$ and $g_3$ to $g_4$ by a $2-$path covering of $\Q_n^{top}-\{r_1,r_2\}$ and $g_5$ to
$g_6$ by a Hamiltonian path of $\Q_n^{bot}-\{r_3\}.$ That produces the desired $3-$path covering of $\Q_n- \F$ for this
case.

\emph{Case 4.}  $g_1,g_2,g_3,g_5$ are on the top plate and $g_4, g_6$ are on the bottom plate.

Let $x$ be a letter such that $g_3x v \neq g_4,g_6,$ and let $g$ be any green vertex on the top plate such that $gv \neq
r_3.$ Let $(g_1, \xi), (g_5, \eta)$ be a $2-$path covering of $\Q_n^{top}-\{r_1,r_2,g_3,g_3x\}$ that connects $g_1$ to
$g_2$ and $g_5$ to $g.$ Such path covering exists since $[4,2,0,2]=5.$ Let $(g_3xv,\mu),(gv,\nu)$ be a $2-$path covering
of $\Q_n^{bot}-\{r_3\}$ that connects $g_3 xv$ to $ g_4$ and $gv$ to $g_6.$ Such path covering exists since
$[1,1,1,1]=4.$ The desired $3-$path covering of $\Q_n- \F$ for this case is $(g_1,\xi), (g_3, xv\mu), (g_5, \eta v \nu
).$

\emph{Case 5.}  $g_1,g_2,g_3,$ are on the top plate and $g_4,g_5, g_6$ are on the bottom plate.

Let $g$ be a green vertex on the top plate such that $gv\neq r_3.$ Let $(g_1, \xi), (g_3, \eta)$ be a $2-$path covering
of $\Q_n^{top}-\{r_1,r_2\}$ that connects $g_1$ to $g_2$ and $g_3$ to $g.$ Such path covering exists since $[2,2,0,2]=
4.$ Let $(gv, \mu), (g_5, \nu)$ be a $2-$path covering of $\Q_n^{bot}-\{r_3\}$ that connects $gv$ to $g_4$ and $g_5$ to
$g_6.$ Such path covering exists since $[1,1,1,1]= 4.$ The desired $3-$path covering of $\Q_n- \F$ for this case is
$(g_1,\xi), (g_3, \eta v\mu), (g_5, \nu ).$

\emph{Case 6.}  $g_1,g_3,g_5$ are on the top plate and $g_2,g_4,g_6$ are on the bottom plate.

Without loss of generality we can assume that $g_5v \neq r_3.$ Since $g_1$ and $g_3$ together have at least eight
neighbors in $\Q_n^{top}$ (Lemma \ref{neighborscardinality}) and there are only two deleted red vertices on the top
plate and three green terminals on the bottom plate, we can also assume, renumbering $g_1$ and $g_3$, if necessary, that
there is a letter $x$ such that $g_3xv$ is not a terminal and $g_3x$ is not a deleted vertex. Finally, let $y$ be a
letter such that $g_2 y \ne r_3$ and $g_2 y v $ is not a terminal. Let $(g_1, \eta)$ be a Hamiltonian path of
$\Q_n^{top}-\{r_1,r_2,g_3,g_3x,g_5\}$ that connects $g_1$ to $g_2yv.$ Such path exists since $[5,1,0,1]=5.$ Let $(g_3xv,
\theta), (g_5v, \zeta)$ be a $2-$path covering of $\Q_n^{bot}-\{r_3, g_2,g_2y\}$ that connects $g_3xv$ to $g_4$ and
$g_5v$ to $g_6.$ Such path exists since $[3,1,1,1]=5.$ The desired $3-$path covering of $\Q_n-\{r_1,r_2,r_3\}$ for this
case is $(g_1, \eta v y), (g_3, x v \theta), (g_5, v \zeta).$

\emph{Case 7.} $g_1, g_2$ are on the top plate and $g_3,g_4,g_5,g_6$ are on the bottom plate.

Let $x,y$ be letters such that neither $g_5 xv$ nor $g_6 yv $ is a terminal vertex. Let $(g_1, \xi),$ $(g_5 xv, \eta)$
be a $2-$path covering of $\Q_n^{top}-\{r_1,r_2\}$ that connects $g_1$ to $g_2$ and $g_5xv $ to $g_6 yv$
($[2,2,0,2]=4$). Let $(g_3,\mu)$ be a Hamiltonian path of $\Q_n^{bot}-\{r_3,g_5,g_5x,g_6,g_6y\}$ that connects $g_3$ to
$g_4.$ Such path exists since $[5,1,0,1]=5.$ The desired $3-$path covering of $\Q_n-\{r_1,r_2,r_3\}$ for this case is
$(g_1, \xi),$ $(g_3, \mu),$ $(g_5, x v \eta v y).$

\emph{Case 8.} $g_1, g_3$ are on the top plate and $g_2,g_4,g_5,g_6$ are on the bottom plate.

Let $x$ be a letter such that $g_4x\ne r_3$ and $g_4xv$ is not a terminal, and let $g$ be any green vertex on the top
plate such that $gv \neq r_3.$ Let $(g_1, \xi), (g_3, \eta)$ be a $2-$path covering of $\Q_n^{top}-\{r_1,r_2\}$ that
connects $g_1$ to $g$ and $g_3$ to $g_4 x v$ ($[2,2,0,2]=4$). Let $(gv,\mu),(g_5,\nu)$ be a $2-$path covering of
$\Q_n^{bot}-\{r_3, g_4, g_4 x\}$ that connects $gv$ to $g_2$ and $g_5$ to $g_6.$ Such path covering exists since
$[3,1,1,1]=5.$ The desired $3-$path covering of $\Q_n-\{r_1,r_2,r_3\}$ for this case is $(g_1, \xi v \mu),$ $(g_3, \eta
v x ),$ $(g_5,\nu).$

\emph{Case 9.} $g_1 $ is on the top plate and $g_2,g_3,g_4,g_5,g_6$ are on the bottom plate.

Assume that there exists a letter $x$ such that $g_2 x = g_1 v \neq r_3.$  Let $(g_3, \xi),$ $(g_5, \eta)$ be any
$2-$path covering of $\Q_n^{bot}-\{r_3,g_2x\}$ that connects $g_3$ to $g_4$ and $g_5$ to $g_6.$ Such path covering
exists since $[2,2,0,2]=4.$ Without loss of generality we can assume that $g_2$ lies on the path connecting $g_3$ to
$g_4.$ Let $\xi = \mu \nu$ with $g_3 \mu = g_2$ and let $(g_3 \mu' v, \zeta)$ be a Hamiltonian path of
$\Q_n^{top}-\{r_1,r_2,g_1\}$ that connects $g_3 \mu'v $ to $g_2 \varphi(\nu) v$ ($[3,1,0,1]=4$). The desired $3-$path
covering of $\Q_n-\{r_1,r_2,r_3\}$ for this case is $(g_1, v x), (g_3, \mu' v \zeta v \nu^* ), (g_5,\eta).$

If $g_1 v = r_3$ or if the distance from $g_1$ to $g_2$ is greater than $2$ we let $x$ be any letter such that $g_2 x
\neq r_3.$ Let $(g_3, \xi), (g_5, \eta)$ be any $2-$path covering of $\Q_n^{bot}-\{r_3,g_2x\}$ that connects $g_3$ to
$g_4$ and $g_5$ to $g_6$ ($[2,2,0,2]=4$). Without loss of generality we can assume that $g_2$ lies on the path
connecting $g_3$ to $g_4.$ Let $\xi = \mu \nu$ with $g_3 \mu = g_2$ and let $(g_1, \theta), ( g_3 \mu' v, \zeta)$ be a
$2-$path covering of $\Q_n^{top}-\{r_1,r_2\}$ that connects $g_1$ to $g_2xv$ and $g_3 \mu'v $ to $g_2 \varphi(\nu) v$
($[2,2,0,2]=4$). The desired $3-$path covering of $\Q_n-\{r_1,r_2,r_3\}$ for this case is $(g_1, \theta v x), (g_3, \mu'
v \zeta v \nu^* ), (g_5,\eta).$

\emph{Case 10.} All the green terminals are on the bottom plate.

Let $x$ and $y$ be any letters different from $v.$ Let $(g_1,\xi)$ be a Hamiltonian path of $Q_n^{bot}-\{r_3, g_5, g_5x,
g_6, g_6y \}$ that connects $g_1$ to $g_2.$ Such path exists since $[5,1,0,1]=5.$  We can assume that $\xi = \eta \theta
\zeta $ with $ g_3 = g_1 \eta, g_4 = g_3 \theta,$ by renumbering $g_3$ and $g_4$, if necessary. Let $(g_5 x v, \mu),
(g_1, \eta' v, \nu)$ be a $2-$path covering of $\Q_n^{top}-\{r_1,r_2\}$ that connects $g_5xv$ to $g_6yv$ and $g_1 \eta'
v$ to $ g_4 \varphi(\zeta)v$ ($[2,2,0,2]=4$). The desired $3-$path covering of $\Q_n-\{r_1,r_2,r_3\}$ for this case is
$(g_1, \eta' v \nu v  \zeta^* )$, $(g_3, \theta)$, $(g_5,xv \mu v y).$
\end{proof}

The following corollary follows directly from Lemma \ref{threecoveringlemma}, Lemma \ref{threechargedpaths}, and Lemma
\ref{neutraltocharged}.

\begin{corollary}\label{c0030}
$5 = [0,0,3,0] \leq [1,1,2,1] \leq [2,2,1,2] \leq 6$.
\end{corollary}

\begin{corollary} $[0,0,1,2] \leq 6$,
$5\le [1,1,0,3] \leq 6$, and $[5,1,1,1]\geq 5.$\footnote{While this paper was under review the authors were
able to prove that $[0,0,1,2] = 4$ (\cite{CGGL1}), $[1,1,0,3]=5$, 
$[1,1,2,1]=5$ (\cite{CGGL3}), $[4,0,2,0]=5$ and $[7,1,0,1]=6$ 
(\cite{CGGL2}).}
\end{corollary}

\begin{proof}
The upper bounds of the first two inequalities follow directly from Corollary \ref{c0030} and Lemma
\ref{redchargedtoneutral}. The last inequality follows from Lemma \ref{greenchargedtoneutral} and the fact that
$[4,2,0,2]=5$ (Lemma \ref{l4202}). The following counterexample shows that $[1,1,0,3]\ge 5$.

Let $n=4$ and $r=(0,1,1,0)$. Let also $r_1 = (0,0,1,1)$, $r_2 = (0,1,0,1)$, $g_1 = (1,0,1,1)$, $g_2 = (1,1,1,0)$, $g_3 =
(1,1,0,1)$, and  $g_4=(1,0,0,0)$ be vertices in $\Q_4-\{r\}$. Then one can directly verify that a $3-$path covering of
$\Q_4-\{r\}$ with paths connecting $r_1$ to $r_2$, $g_1$ to $g_2$, and $g_3$ to $g_4$ does not exist.
\end{proof}

\begin{lemma}\label{locke4}$([8]=6)$ 
Let $ n \ge 6 $ and $\F$ be any neutral fault of mass eight in $\Q_n.$ 
Then $\Q_n - \F$ is Hamiltonian.
\end{lemma}

\begin{proof} Let $n\ge 6.$ We split $\Q_n$ into two plates so that
each plate has at least one  red deleted vertex. There are two general cases: \emph{Case A} -- there are two red deleted
vertices on each plate and \emph{Case B} -- there are three red deleted vertices on the top plate and one red deleted
vertex on the bottom plate. Within each general case there are subcases that depend on the distribution of the green
deleted  vertices on the plates.

Let the fault be $\F=\{r_1, r_2, r_3, r_4, g_1, g_2, g_3, g_4\}$ with the $r_i$ red and the $g_i$ green.

\emph{Case A.} $r_1, r_2$ are on the top plate and $r_3, r_4$ are on the bottom plate.

\emph{Subcase A1.} All the green deleted vertices are on the top plate.

Let $(g_1, \xi),$ $(g_2, \eta) $ be a $2-$path covering of $\Q_n^{top}-\{r_1, r_2\}$ that connects $g_1$ to $g_3$ and
$g_2$ to $g_4.$ Such path covering exists since $[2,2,0,2]=4.$ Let $(g_1 \xi'v, \mu),$ $(g_1 \varphi(\xi) v, \nu)$ be a
$2-$path covering of $\Q_n^{bot}-\{r_3,r_4\}$ that connects $g_1\xi' v$ to $g_2 \eta' v$ and $g_1 \varphi(\xi)v$ to
$g_2\varphi(\eta)v.$ Such path covering exists since $[2,2,0,2]=4.$ The desired Hamiltonian cycle for this case is 
$(g_1 \varphi(\xi), \xi '^* v \mu v (\eta'^*)^R v \nu^R v).$

\emph{Subcase A2.} $g_1,g_2,g_3$ are on the top plate and $g_4$ is on the bottom plate.

Let $r_5, r_6$ be any two non-deleted red vertices on the top plate such that neither $r_5v$ nor $r_6v$ is a deleted
vertex. Let $(r_5,\xi)$ be a Hamiltonian path of $\Q_n^{top}-\{r_1,r_2,g_1,g_2,g_3\}$ that connects $r_5$ to $r_6.$ Such
path exists since $[5,1,0,1]=5.$ Let $(r_6v,\eta)$ be a Hamiltonian path of $\Q_n^{bot}-\{r_3,r_4,g_4\}$ that connects
$r_6v$ to $r_5v.$ Such path exists since $[3,1,0,1]=4.$ The desired Hamiltonian cycle for this case is $(r_5, \xi v \eta
v).$

\emph{Subcase A3.} $g_1,g_2$ are on the top plate and $g_3, g_4$ are on the bottom plate.

Let $r, g$ be a red and a green non-deleted vertices on the top plate such that neither $rv$ nor $gv$ is a deleted
vertex. Let $(r,\xi)$ be a Hamiltonian path of $\Q_n^{top}-\{r_1,r_2,g_1,g_2\}$ that connects $r$ to $g.$ Such path
exists since $[4,0,1,0]=5.$ Let $(gv , \eta)$ be a Hamiltonian path of $\Q_n^{bot}-\{r_3,r_4,g_3,g_4\}$ that connects
$gv$ to $rv.$ Such path exists since   The desired Hamiltonian cycle for this case is $(r, \xi v \eta v).$

\emph{Case B.} $r_1, r_2, r_3$ are on the top plate and $r_4$ is on the bottom plate.

\emph{Subcase B1.} All the green deleted vertices are on the top plate.

Let $(g_1, \xi)$ be a Hamiltonian path for $\Q_n^{top}-\{r_1,r_2,r_3,g_3,g_4\}$ that connects $g_1$ to $g_2.$ Such path
exists since $[5,1,0,1]=5.$ Let $(g_1 \xi' v,\eta)$ be a Hamiltonian path of $\Q_n^{bot}-\{r_4\}$ that connects $g_1\xi'
v$ to $g_1 \varphi(\xi)v.$ Such path exists since $[1,1,0,1]=2.$ The desired Hamiltonian cycle for this case is $(g_1
\varphi(\xi), \xi'^* v \eta v).$

\emph{Subcase B2.} $g_1,g_2,g_3$ are on the top plate and $g_4$ is on the bottom plate.

Let $\gamma $ be any Hamiltonian cycle of $\Q_n^{top}-\{r_1,r_2,r_3,g_1,g_2,g_3\}.$ Such cycle exists since $[6]=5.$ We
can find a vertex $g$ on this cycle such that $\gamma =(g,\xi)$ with neither $gv$ nor $g\xi'v$ being a deleted vertex.
Let $(g\xi'v, \eta)$ be a Hamiltonian path of $\Q_n^{bot}-\{r_4,g_4\}$ that connects $g\xi'v$ to $gv.$ The desired
Hamiltonian cycle for this case is $(g, \xi' v \eta v).$

\emph{Subcase B3.} $g_1$ is on the top plate and $g_2,g_3,g_4$ are on the bottom plate.

Let $g_5, g_6, g_7, g_8$ be any green non-deleted vertices on the top plate such that none of $g_5v, g_6v, g_7v ,g_8v$
is a deleted vertex. Let $(g_5, \xi), (g_7, \eta)$ be a $2-$path covering of $\Q_n^{top}-\{r_1,r_2,r_3,g_1\}$ that
connects $g_5$ to $g_6$ and $g_7$ to $g_8.$ Such path covering exists since $[4,2,0,2]=5.$ Let $(g_6v, \mu ), (g_8,\nu)$
be a $2-$path covering of $\Q_n^{bot}-\{r_4,g_2,g_3,g_4\}$ that connects $g_6v$ to $g_7v$ and $g_8v$ to $g_5v.$ Such
path covering exists since $[4,2,0,2]=5.$ The desired Hamiltonian cycle for this case is $(g_5,\xi v \mu v \eta v \nu
v).$

\emph{Subcase B4.} $g_1,g_2$ are on the top plate and $g_3, g_4$ are on the bottom plate.

This case is equivalent to \emph{Subcase A2.}

\emph{Subcase B5.} All the green deleted vertices are on the bottom plate.

This case can be avoided if $n=6.$ Indeed, if the four deleted red vertices are contained in a three dimensional subcube
of $\Q_6$ then we can split $\Q_6$ into two plates with $2$ deleted red vertices on each plate. If the four deleted red
verctices are not contained in any three dimensional subcube of $\Q_6$ then there are at least $4$ coordinates that
split the red vertices. At least one of these coordinates must split the green deleted vertices as well, for otherwise
the $4$ green deleted vertices would have to be contained in a two dimensional subcube which is impossible. Therefore,
for this case we assume that $n \ge 7.$

Let $(g_1, \xi )$ be a Hamiltonian path of $\Q_n^{bot}-\{r_4\}$ that connects $g_1$ to $g_4.$ Such path exists since
$[1,1,0,1]=2.$ It follows from Lemma \ref{separatinglemma}, renumbering $g_2$ and $g_3$, if necessary, that $\xi=
\eta\theta\zeta$ with $g_1 \eta = g_2$, $g_2 \theta = g_3$, $g_3 \zeta = g_4$ and the words $\eta$, $\theta$, and
$\zeta$ each of length at least $4.$ Let $(g_1\varphi (\eta) v, \kappa)$, $(g_3 \varphi (\zeta) v, \mu)$, and $(g_2
\varphi(\theta)v, \nu)$ be a $3-$path covering of $\Q_n^{top}-\{r_1,r_2,r_3\}$ that connects $g_1\varphi(\eta)v$ to $g_1
\xi'v,$ $g_3 \varphi(\zeta)v$ to $g_2\theta'v,$ and $g_2 \varphi(\theta)v$ to $g_1 \eta' v.$ The existence of such path
covering follows from Lemma \ref{threechargedpaths}. The desired Hamiltonian cycle for this case is $(g_1 \varphi(\eta),
v \kappa v (\zeta'^*)^R v \mu v (\theta'^*)^R v \nu v(\eta'^*)^R.$
\end{proof}

\section{Concluding Remarks and Conjectures}

We have found several values of $ [M,C,N,O]$  when the parameters involved are relatively small. Unfortunately, as the
parameters increase the number of cases to be considered in the proofs becomes extremely large. We hope that further
analysis and improvement of our proofs will lead to substantial simplifications. Our results support the following
conjectures:

\begin{conjecture}[Locke \cite{locke}]\label{C1} Let $k \ge 0.$ Then $[2k] = k+2.$
\end{conjecture}

We have already discussed that $[2k] \ge k+2.$ And after this paper we know that 
the conjecture is true for $ 0 \le k \le 4.$ The proof of this conjecture for 
$k\ge 5$, which depends on the proof of Conjecture \ref{chargeoneonepathconjecture}, is contained in \cite{castanedagotchev4}.

\begin{conjecture} \label{chargeoneonepathconjecture} Let $k \ge 1.$ Then $[2k+1,1,0,1]=k+3.$
\end{conjecture}

In this article we have proved that this conjecture is true for $k=1$ and $k=2$ 
and the proof for the case $k=3$ is contained in \cite{CGGL2}. The proof of this 
conjecture for $k\ge 4$, which depends on the proof of Conjecture \ref{C1}, is 
contained in \cite{CGGL4}. Here we can show that $[2k+1,1,0,1]
\ge k+3.$ Indeed, let $r$ be any red vertex in $\Q_{k+2}$ and $\F$ be a fault of mass $2k+1$ that  contains any $k+1$
red vertices different from  $r$ and all the green vertices adjacent to $r$ except two vertices $g_1$ and $g_2.$ Then,
obviously, the only path in $ \Q_{k+2} - \F$ that connects $g_1$ to $g_2$ and visits $r$ is of length $3$ and cannot be
a Hamiltonian path of $ \Q_{k+2} - \F$ if $k\ge 1.$

The following conjecture is a direct corollary of Conjecture \ref{chargeoneonepathconjecture}.

\begin{conjecture} \label{chargezeroonepathconjecture} Let $k \ge 1.$
Then $[2k,0,1,0]=k+3.$
\end{conjecture}

In this article we have proved this conjecture for $k=1$ and $k=2.$ Let us prove that $[2k,0,1,0] \ge k+3.$

Let  $x_1, x_2, \dots, x_{k+2}$ be the standard generators of $ {\bf Z}_2^{k+2}.$ We select any red vertex $r$ in
$\Q_{k+2}$ and set
$$\F = \{rx_1, rx_2, \dots, rx_k, rx_{k+2}x_1, rx_{k+2}x_2, \dots,
rx_{k+2}x_{k} \}.$$ Then the only path that connects $rx_{k+1}$ to $rx_{k+1}x_{k+2}$ and visits $r$ is of length $3$
and cannot be a Hamiltonian path of $ \Q_{k+2} - \F$ if $k\ge 1.$

\begin{conjecture} \label{chargeonetwopathsconjecture}
Let $k \ge 0.$ Then   $[2k+1,1,1,1]=k+4.$
\end{conjecture}

In this article we have proved this conjecture for $k=0,1.$  Let us prove that $[2k+1,1,1,1] \ge k+4.$

Let  $\{x_1, x_2, \dots, x_{k+3}\}$ be the standard generators of $ {\bf Z}_2^{k+3}.$ We select any red vertex $r$ in
$\Q_{k+3}$ and set
$$\F = \{rx_1, rx_2, \dots, rx_k, rx_{k+3}x_1, rx_{k+3}x_2, \dots,
rx_{k+3}x_{k+1} \}.$$ Then there does not exist a $2-$path covering of $\Q_{k+3} - \F $ that connects $rx_{k+1}$ to
$rx_{k+2}$ and $ r x_{k+2}x_{k+3}$ to any green vertex $g\notin \F$ for $r$ and $rx_{k+3}$ are blocked between all
deleted and terminal vertices.

Even though our main focus in this article is the production of path coverings with prescribed ends for the hypercube
with or without deleted vertices, we occasionally have considered the more general problem of prescribing ends and
edges. The following conjecture is related to this problem.

\begin{conjecture} Let $k \ge 0$ and $n \ge k+4.$ Let also $\F$
be any fault in $\Q_n$ with $k+1$ red vertices and $k$ green vertices, $g_1$ and $g_2$ be two green vertices in $ \Q_n -
\F,$ and $e=\{a,b\}$ be any edge different from $\{g_1,g_2\}$ and not incident to any of the vertices of $\F.$ Then
there exists a Hamiltonian path of $\Q_n - \F$ that connects $g_1$ to $g_2$ and passes through the edge $e.$
\end{conjecture}

In this article we have proved this conjecture for $ k =0.$ To see that $n \ge k+4$, assume that $n=k+3,$ and  let $r$
and $\F$ be selected as in the discussion of Conjecture \ref{chargeonetwopathsconjecture}. Let $g_1 = r x_{k+1}, g_2 =
rx_{k+2} ,$ and $e= \{g_2, rx_{k+3} x_{k+2}\}.$ Then the only path in $\Q_n - \F $ that connects $g_1$ to $g_2,$ passes
through $e,$ and visits $ rx_{k+3}$ is $ g_2, rx_{k+3}x_{k+2}, rx_{k+3}, r, g_1$ which obviously is not a Hamiltonian
path of $\Q_{k+3}- \F.$

Finally, we point out that in \cite{castanedagotchev3} we use results from this article to obtain the following
generalization of a theorem  of Fu \cite{fu}:

\begin{theorem}[\cite{castanedagotchev3}] Let $f$ and $n$ be integers with $n \ge 5$ and $0 \le f \le 3n-7.$
Then for any set of vertices $\F$ of cardinality $f$ in $\Q_n$ there exists a cycle in $\Q_n- \F$ of length at least
$2^n - 2f.$
\end{theorem}

\begin{appendix}
\section{$2-$path coverings of $\Q_4$}\label{2pathcoveringsofQ4}

When a neutral pair is deleted from $\Q_4$ one can still freely prescribe the ends for a $2-$path covering of the
resulting graph. In spite the fact that the dimension is so low we find it difficult to verify this statement by
inspection. Therefore, we provide a proof below for the benefit of the reader.

\begin{lemma} \label{neutralpathsinq4} Let $\F=\{r,g\}$ be a neutral fault in $\Q_4,$ and
$r_1,r_2,g_1,g_2$ be two red and two green vertices in $\Q_4-\F.$ Then there exists a $2-$path covering of $\Q_4 -\F$
with one path connecting $r_1$ to $g_1$ and the other connecting $r_2$ to $g_2.$
\end{lemma}

\begin{proof} The deleted vertices $r$ and $g$ have opposite parity and
belong to $\Q_4$. Therefore we can split $\Q_4$ in such way that both vertices belong to the same plate, say
$\Q_4^{top}$. We consider all essentially different cases that depend on the distribution of the vertices
$r_1,r_2,g_1,g_2$ between the plates.

\emph{Case 1.} $r_1,r_2,g_1,g_2 \in \Q_4^{top}$.

\emph{Subcase 1(a).} Let  $\{r_1,g_1\}, \{r_2,g_2\} \in \mathcal{B}_{\{r,g\}}.$ Then there exists a one-letter word $x$
such that $(r_1,x), (g_1,x)$ is a $2-$path covering of $ \Q_4^{top} -\{r,g,r_2,g_2\}.$ Let $(r_1 x v, \mu), (r_2 v,
\nu)$ be a $2-$path covering of $\Q_4^{bot}$ that connects $r_1 xv$ to $ g_1 x v,$ and $r_2 v$ to $g_2v.$ Such path
covering exists since $[0,0,2,0]=2.$ The desired $2-$path covering of $\Q_4 -\{r,g\}$ is $(r_1, x v \mu v x),$ $(r_2, v
\nu v ).$

\emph{Subcase 1(b).} If either $\{r_1,g_1\}$ or $\{r_2,g_2\}$ is not in $\mathcal{B}_{\{r,g\}}$ we can assume without
loss of generality that  $\{r_1,g_1\} \not \in \mathcal{B}_{\{r,g\}}.$ Then, according to Lemma
\ref{q3minustwovertices}(1), there exists a Hamiltonian path $(r_1, \xi)$ of $\Q_4^{top}-\{r,g\}$ that connects $r_1$ to
$g_1$. Let $ \xi = \eta \theta \zeta $ with $ (r_1 \eta, r_1 \eta \theta)$ equals $(r_2, g_2) $ or $(g_2,r_2).$ Let also
$ (r_1 \eta' v, \mu)$ be a Hamiltonian path of $\Q_4^{bot}$ that connects $ r_1 \eta' v$ to $ g_1 (\zeta^R)'v.$ The
desired $2-$path covering of $\Q_4 - \{r,g\}$ is $(r_1, \eta' v \mu v \zeta^*), (r_1\eta, \theta).$

\emph{Case 2.} $r_1,r_2,g_1$ are on the top plate and $g_2$ is on the bottom plate.

\emph{Subcase 2(a).} If $\{g_1,r_2\} \not \in \mathcal{B}_{\{r,g\}}$ then, according to Lemma
\ref{q3minustwovertices}(1), there exists a Hamiltonian path $(g_1, \omega )$ of $ \Q_4^{top}- \F $ that connects $ g_1$
to $r_2.$ Let $\omega= \xi \eta $ with $ g_1 \xi = r_1$ and let $ ( r_1 \varphi(\eta)v, \theta)$ be a Hamiltonian path
of $ Q_4^{bot}$ that connects $ r_1 \varphi(\eta)v$ to $ g_2.$ The desired $2-$path covering of $\Q_4 - \{r,g\}$ is
$(g_1, \xi), (g_2,\theta^R v \eta^*).$

\emph{Subcase 2(b).} If $\{g_1,r_2\} \in \mathcal{B}_{\{r,g\}}$ then $\{g_1,r_1\} \not \in \mathcal{B}_{\{r,g\}}$ and
there exists a Hamiltonian path $ (g_1, \omega)$ of $\Q_4 - \F$ that connects $ g_1$ to $ r_1.$ Let $\omega =\xi \eta $
with $ g_1 \xi = r_2.$ We have to consider two sub-subcases:

(i) $g_2v = r_1 $ or $ g_2v = r_2.$

We observe that the lengths of $\xi$ and $\eta$ are $ 1$ and $4$ or $ 3$ and $2.$

If $\xi $ is the longer word, then we use $[0,0,2,0]=2$ to produce a $2-$path covering $(g_1 \xi'v, \mu), (g_1 \xi ''v,
\nu)$ of $\Q_4^{bot}$ that connects $ g_1 \xi' v$ to $g_2$ and $ g_1 \xi''v$ to $ r_2 \varphi(\eta)v.$ The desired
$2-$path covering of $\Q_4 - \{r,g\}$ is $(g_1, \xi''v\nu v \eta^*), (g_2,\mu^R v \varphi(\xi^R)).$

If $\eta$ is the longer word, then we use $[0,0,2,0]=2$ to produce a $2-$path covering $(g_1 \xi'v, \mu), (r_2
\varphi(\eta)v , \nu)$ of $\Q_4^{bot}$ that connects $ g_1 \xi' v$ to $r_1 (\eta^R)''v$ and $ r_2 \varphi(\eta)v $ to
$g_2.$ The desired $2-$path covering of $\Q_4 - \{r,g\}$ is $(g_1, \xi'v\mu v \eta^{**}), (r_2, \varphi(\eta)v\nu ).$

(ii) $g_2v$ is an interior vertex of the path $ (g_1, \omega).$

If $ \xi = \theta \zeta$ with $ g_1 \theta = g_2 v$ then we use $[1,1,0,1]=2$ to produce a Hamiltonian path $(g_1
\theta' v , \mu)$ of $ \Q_4^{top} - \{g_2\}$ that connects $ g_1\theta' v$ to $ r_2 \varphi(\eta)v.$ The desired
$2-$path covering of $\Q_4 - \{r,g\}$ is $(g_1, \theta' v \mu v \eta^*  ), (r_2, \zeta^R v  ).$

If $\eta = \theta \zeta $ with $ r_2 \theta = g_2 v $ then we use $[1,1,0,1]=2$ to produce a Hamiltonian path $(g_1 \xi'
v , \mu)$ of $ \Q_4^{top} - \{g_2\}$ that connects $ g_1\xi' v$ to $ r_1 (\zeta^R)'v.$ The desired $2-$path covering of
$\Q_4 - \{r,g\}$ is $(g_1, \xi' v \mu v \zeta^* ), (r_2, \theta v).$

\emph{Case 3.} $r_1,r_2 \in \Q_4^{top}$ and $g_1,g_2 \in \Q_4^{bot}$.

Find a Hamiltonian path of $\Q_4^{top}-\{g\}$ that connects $r_1$ to $r_2$. The vertex $r$ belongs to that path. Cut
that path just before $r$ and right after $r$ and connect these two vertices with bridges to the bottom plate. Let $r_3$
and $r_4$ be the ends of these bridges that belong to the bottom plate. Then use $[0,0,2,0]=2$ to find a $2-$path
covering of the bottom plate that connects $r_3$ and $r_4$ to the appropriate vertices $g_1$ and $g_2$.

\emph{Case 4.} $r_1,g_1 \in \Q_4^{top}$ and $r_2,g_2 \in \Q_4^{bot}$.

Consider $\Q_4^{top}$. It is not difficult to verify that either there is a Hamiltonian path for $\Q_4^{top}-\{r,g\}$
that connects $r_1$ to $g_1$ or there is a path with length $3$ connecting $r_1$ to $g_1$ such that exactly one edge
remains not covered. In the first case use $[0,0,1,0]=1$ to find a Hamiltonian path for $\Q_4^{bot}$ connecting $r_2$ to
$g_2$. In the second case denote by $r_3$ and $g_3$ the vertices in the bottom plate that are neighbors of the vertices
in $\Q_4^{top}$ that are not covered. Use Corollary \ref{refinedhavel} to find a Hamiltonian path for $\Q_4^{bot}$ that
connects $r_2$ to $g_2$ and passes trough the edge $\{r_3,g_3\}$. Cut that path at that edge and using two bridges
connect both pieces to the non-covered edge from the top plate.

\emph{Case 5.} $r_1,g_2 \in \Q_4^{top}$ and $r_2,g_1 \in \Q_4^{bot}$.

We consider two subcases:

\emph{Subcase 5(a).} Assume that $\{r_1,g_2\} \not \in \mathcal{B}_{\{r,g\}}.$ Then there is a Hamiltonian path $ (r_1,
\xi)$ of $ \Q_4^{top}- \{r,g\}$ that connects $r_1$ to $g_2.$ There are three sub-subcases that depend on whether or not
$r_2$ or $g_1$ are adjacent to vertices inside of the path $(r_1, \xi).$

(i) Assume that $\xi = \eta \theta$ with $r_1 \eta v= g_1.$ Let $(r_1 \eta \varphi(\theta)v, \mu) $ be a Hamiltonian
path of $\Q_4^{bot}- \{g_1\}$ that connects $r_1 \eta \varphi(\theta)v$ to $ r_2.$ The desired $2-$path covering of $
\Q_4 - \F$ for this case is $(r_1, \eta v), (r_2, \mu^R v \theta^*).$

(ii) Assume that $ \xi = \eta \theta$ with $g_2 \theta^R v= r_2.$ Let $(r_1 \eta'v, \mu) $ be a Hamiltonian path of
$\Q_4^{bot}- \{r_2\}$ that connects $r_1 \eta' v$ to $ g_1.$ The desired $2-$path covering of $ \Q_4 - \F$ for this case
is $(r_1, \eta' v \mu), (r_2, v \theta).$

(iii) Finally, let neither $r_2$ nor $g_1$ be adjacent to a vertex in the path $(r_1,\xi).$ Let $\xi = x y \eta $ for
some letters $x,y,$ and a word $\eta.$ Then there is a $2-$path covering $(r_1xv, \mu), (r_1 xyv,\nu)$ of $\Q_4^{bot}$
that connects $ r_1xv$ to $ g_1$ and $ r_1 xy v$ to $r_2.$ The desired $2-$path covering of $ \Q_4 - \F$ for this case
is $(r_1,xv  \mu), (r_2, \nu^R v\eta).$

\emph{Subcase 5(b).} Let $\{r_1,g_2\} \in \mathcal{B}_{\{r,g\}}.$ Then, according to Lemma \ref{q3minustwovertices},
there exist two distinct $2-$path coverings of $\Q_4^{top}-\{r,g\}$ with paths of length $2,$ one starting at $r_1$ and
the other starting at $g_2.$ We can choose a $2-$path covering of $\Q_4^{top}-\{r,g\}$ to be $(r_1, \xi), (g_2,\eta),$
with $r_1 \xi v \neq g_1$ or $ g_2 \eta v \neq r_2.$ There are three sub-subcases:

(i) Let $r_1 \xi v \neq g_1$ and $ g_2 \eta v \neq r_2.$ Let $ (r_1 \xi v, \mu), (g_2 \eta v, \nu)$ be a $2-$path
covering of $\Q_4^{bot}$ that connects $ r_1 \xi v $ to $ g_1$ and $ g_2 \eta v $ to $ r_2.$ The desired $2-$path
covering of $ \Q_4 - \F$ for this case is $(r_1,\xi v  \mu), (g_2, \eta v \nu).$

(ii) Let $r_1 \xi v \neq g_1$ and $ g_2 \eta v = r_2.$ Let $ (r_1 \xi v, \mu)$ be a Hamiltonian path  of $\Q_4^{bot}-
\{r_2\}$ that connects $ r_1 \xi v $ to $ g_1.$ The desired $2-$path covering of $ \Q_4 - \F$ for this case is $(r_1,\xi
v  \mu), (g_2, \eta v ).$

(iii)  Let $r_1 \xi v = g_1$ and $ g_2 \eta v \neq r_2.$ This case is completely symmetrical to case (ii).

\emph{Case 6.} $r_1 \in \Q_4^{top}$ and $r_2,g_1,g_2 \in \Q_4^{bot}$.

Use Lemma \ref{twodifferentedges} to find a Hamiltonian path of $\Q_4^{top}-\{g\}$ that connects $r$ to $r_1$ and such
that the vertex $g_3$ which is next to $r$ in this path is not adjacent to $r_2$. Let the second end of the bridge that
begins at $g_3$ be $r_3$. Use $[0,0,2,0]=2$ to find a $2-$path covering of the bottom plate that connects $r_3$ to $g_1$
and $r_2$ to $g_2$.

\emph{Case 7.} $r_1,r_2,g_1,g_2 \in \Q_4^{bot}$.

Use $[0,0,2,0]=2$ to find a $2-$path covering of $Q_4^{bot}$ that connects $r_1$ to $g_1$ and $r_2$ to $g_2$. Then find
an edge that belongs to one of the two paths whose neighbors $r_3$ and $g_3$ in $Q_4^{top}$ are not deleted vertices and
also $\{r_3,g_3\} \not \in \mathcal{B}_{\{r,g\}}.$ Cut that path at that edge and use Lemma \ref{q3minustwovertices} to
find a Hamiltonian path for $Q_4^{top}-\{r,g\}$ that connects $r_3$ to $g_3$.
\end{proof}

\section{Summary of results}\label{B}

The following table summarizes some of the results obtained in this paper.
The rows represent admissible combinations of $M$ and $C$ and the columns
contain all the values of $N$ and $O$ such that $N+O \le 3$. Each star in the
table represents an impossible case. The missing entries in the table
correspond to values of $[M,C,N,O]$ that we do not know yet. The inequalities
in the table represent an upper or lower bound of the corresponding entry.
Finally, the entries with an asterisk are results that were obtained after
this paper was submitted for publication and therefore their proofs are not
contained in this paper.

\begin{center}
\vskip10pt
\begin{tabular}{||c|c|c|c|c|c|c|c|c|c||}
\hline $MC \backslash NO$ & $01$ & $10$ & $20$ & $11$ & $02$ & $30$ & $21$ & $12$ & $03$ \\ \hline $00$  & $\star$ & $1$
& $2$ & $\star$ & $4$ & $5$ & $\star$ & $4^*$  &  $\star$   \\ $11$  & $2$ & $\star$ & $\star$ & $4$ & $\star$ & $\star$ &
$5^*$  &  $\star$  &  $5^*$\\ $20$  & $\star$ & $4$ & $4$ & $\star$ & $5$ &   & $\star$ &   &  $\star$   \\ $22$  & $\star$
& $\star$ & $\star$ & $\star$ & $4$ & $\star$ & $\star$ & $\le 6$  &  $\star$   \\ $31$  & $4$ & $\star$ & $\star$ & $5$
& $\star$ & $\star$ &  &  $\star$ &     \\ $33$  & $\star$ & $\star$ & $\star$ & $\star$ & $\star$ & $\star$& $\star$ &
$\star$  &  $\le 6$   \\ $40$  & $\star$ & $5$ & $5^*$ & $\star$ & $$ & $$ & $\star$ &   &  $\star$   \\ $42$  & $\star$ &
$\star$ & $\star$ & $\star$ & $5$ & $\star$ & $\star$ &   &  $\star$   \\ $44$  & $\star$ & $\star$ & $\star$ & $\star$
& $\star$ & $\star$ & $\star$ &  $\star$ &  $\star$   \\ $51$  & $5$ & $\star$ & $\star$ & $\geq 5$ & $\star$ & $\star$
& $ $ & $\star$  &  $  $   \\ \hline
\end{tabular}
\end{center}
\end{appendix}
\end{document}